\documentclass[12pt]{iopart}

\usepackage{amssymb,amsfonts,amsthm,iopams}
\def\longrightarrow{
\relbar\joinrel\joinrel\relbar\joinrel\joinrel\relbar\joinrel\joinrel\relbar\joinrel\joinrel\relbar\joinrel\joinrel\relbar\joinrel\joinrel\rightarrow}
\newcommand{\xrightarrow}[1]{\stackrel{#1}{\longrightarrow}}
\newtheorem{theorem}{Theorem}[]
\newtheorem{assumption} [theorem]{Assumption}
\newtheorem{corollary}      [theorem]{Corollary}
\newtheorem{lemma}         [theorem]{Lemma}
\newtheorem*{result*}{Main results}
\usepackage{graphicx, color}
\newcommand{\texorpdfstring}[2]{#1}   % dummy definition of \texorpdfstring
\newcommand{\url}[1]{#1} % dummy definition of \texorpdfstring
\usepackage{hyperref}%
\definecolor{gray}{rgb}{0.2,0.2,.2}
\hypersetup{%
  unicode=true,          % non-Latin characters in Acrobat’s bookmarks
  colorlinks=true,       % false: boxed links; true: colored links
  linkcolor=gray,          % color of internal links
  citecolor=gray      % color of links to bibliography
}
\newcommand{\ul}[1]{\underline{#1}}

\newcommand{\ol}[1]{{\overline{#1}}}

\newcommand{\bigpar}{\par\quad\newline\noindent}
\newcommand{\dint}[1]{\,\mathrm{d}#1}

\newcommand{\fspace}[1]{{\mathsf{#1}}}
\newcommand{\fspaceL}{\fspace{L}}

\newcommand{\Rset}{{\mathbb{R}}}
\newcommand{\Zset}{{\mathbb{Z}}}

\newcommand{\Nset}{{\mathbb{N}}}

\newcommand{\cointerval}[2]{[#1,\,#2)}%
\newcommand{\oointerval}[2]{(#1,\,#2)}%
\newcommand{\ccinterval}[2]{[#1,\,#2]}%

\newcommand{\DO}[1]{{O\at{#1}}}
\newcommand{\Do}[1]{{o\at{#1}}}

\newcommand{\skp}[2]{{\left\langle{#1},\,{#2}\right\rangle}}

\newcommand{\bskp}[2]{{\big\langle{#1},\,{#2}\big\rangle}}

\newcommand{\pair}[2]{{\left({#1},\,{#2}\right)}}

\newcommand{\at}[1]{{\left({#1}\right)}}

\newcommand{\bat}[1]{{\big(#1\big)}}
\newcommand{\Bat}[1]{{\Big(#1\Big)}}

\newcommand{\triple}[3]{{\left({#1},\,{#2},\,{#3}\right)}}

\newcommand{\norm}[1]{\|{#1}\|}
\newcommand{\bnorm}[1]{\big\|{#1}\big\|}

\newcommand{\abs}[1]{\left|{#1}\right|}
\newcommand{\babs}[1]{\big|{#1}\big|}

\newcommand{\ga}{{\gamma}}

\newcommand{\eps}{{\varepsilon}}

\newcommand{\ka}{{\kappa}}
\newcommand{\la}{{\lambda}}
\newcommand{\si}{{\sigma}}

\newcommand{\calC}{\mathcal{C}}

\newcommand{\calP}{\mathcal{P}}

\newcommand{\calT}{\mathcal{T}}

\begin{document}

\title[Asymptotic formulas for solitary waves]{Asymptotic formulas for solitary waves
in the high-energy limit of  FPU-type chains}

\author{Michael Herrmann}

\address{Universit\"at M\"unster, Institut f\"ur Numerische und Angewandte Mathematik,
      \\Einsteinstra{\ss}e 62, D-48149 M\"unster, Germany}
\ead{michael.herrmann@uni-muenster.de}

\author{Karsten Matthies}
\address{University of Bath, Department of Mathematical Sciences,\\
     BA2 7AY Bath, United Kingdom}
\ead{k.matthies@maths.bath.ac.uk}
\vspace{10pt}
\begin{indented}
\item[]\today
\end{indented}
\begin{abstract}
It is well established that the solitary waves of FPU-type chains converge in the high-energy limit to traveling waves of the hard-sphere model. In this paper we establish improved asymptotic expressions for the wave profiles as well as an explicit formula for the wave speed. The key step in our approach is the derivation of an asymptotic ODE for the appropriately rescaled strain profile.
\end{abstract}
\vspace{2pc}
\noindent{\it Keywords}: asymptotic analysis, lattice waves, high-energy-limit of FPU-type chains
%\maketitle
% - content
% -----------------------------------------------------------------------------
%
%
% ------------------------------------------------------------------
\section{Introduction}
% ------------------------------------------------------------------
%
Traveling waves in nonlinear Hamiltonian lattice systems are ubiquitous in many branches of sciences and their mathematical analysis has attracted a lot of interest over the last two decades. In the simplest case of a spatially one-dimensional lattice with nearest-neighbor interactions -- often called Fermi-Pasta-Ulam or FPU-type chain -- the analytical problem consists of finding a positive wave-speed parameter $\si$ along with a distance profile $R$ and a velocity profile $V$ such that
\begin{equation}
\label{Eqn:TW.Diff}
\eqalign{
R^\prime\at{x}
&=
V\at{x+\mbox{$\frac12$}}-
V\at{x-\mbox{$\frac12$}}\,,
\cr
\si \, V^\prime\at{x}
&=\Phi^\prime\Bat{R\at{x+\mbox{$\frac12$}}}-
\Phi^\prime\Bat{R\at{x-\mbox{$\frac12$}}}
}
\end{equation}
is satisfied for all $x\in\Rset$. Here $\Phi$ is the nonlinear interaction potential and the position $u_j\at{t}$ of particle $j$ can be recovered by
\begin{eqnarray*}
u_j\at{t}=U\at{j-\sqrt{\si}\,t},\qquad  U\at{x}:=\int_{x_0}^x V\at{y}\dint{y}\,,
\end{eqnarray*}
which implies the identities
\begin{equation*}
\dot{u}_j\at{t}=\sqrt{\si}\,V\at{j-\sqrt{\si}\,t}\quad{\rm and}\quad u_{j+1}\at{t}-u_j\at{t}=
R\at{j+\mbox{$\frac12$}-\sqrt{\sigma}\,t}
\end{equation*}
for the atomic velocities and distances, respectively. In particular, $u$ satisfies Newton's law of motion
\begin{eqnarray}
\label{Eqn:FPU}
\ddot u_j\at{t}= \Phi^\prime\bat{u_{j+1}\at{t}-u_{j}\at{t}}-
\Phi^\prime\bat{u_{j}\at{t}-u_{j-1}\at{t}}\,, \qquad j \in \Zset\,.
\end{eqnarray}
\par
The existence of several types of traveling wave solutions (with periodic, solitary, front-like, or even more complex profile functions) can be established in different frameworks; see, for instance, \cite{FW96,FV99} for constrained optimization problems, \cite{Pan05} for critical point techniques, \cite{IJ05} for spatial dynamics, and \cite{SZ09,TV14} for almost explicit solutions. However, very little is known about the uniqueness of the solutions to the advance-delay differential equation \eref{Eqn:TW.Diff} or their dynamical stability within \eref{Eqn:FPU}. The only nonlinear exceptions are the Toda chain -- which is completely integrable, see \cite{Tes01} and references therein -- and the Korteweg-de Vries  (KdV) limit of solitary waves in chains with so called hardening. The latter has been investigated by Friesecke and Pego in a series of four seminal papers starting with \cite{FP99}.
In this limit, solitary waves have small amplitudes, carry low energy, and are spread over a huge number of lattice sites. The discrete difference operators in \eref{Eqn:TW.Diff} can therefore be approximated by continuous differential operators and the asymptotic properties are governed by the KdV equation, which is completely integrable and well understood.
\par
Another interesting asymptotic regime concerns solitary waves with high energy in
chains with rapidly increasing potential. Here the profile functions localize completely
since $V$ converges -- maybe after some affine rescaling -- to the indicator function of an interval, see \cite{FM02,Tre04} or \cite{Her10} for potentials $\Phi$ that posses a singularity or grow super-polynomially, respectively. The physical interpretation of the high-energy limit is that the particles interact asymptotically as  in the hard-sphere limit, that means by elastic collisions only.
\par
The high-energy limit is another natural candidate for tackling the analytical problems
concerning the uniqueness and the stability of traveling wave. In this context we are especially interested in the spectral properties of the linearized traveling waves
equation -- see the discussion in \sref{sect:eigenproblem} -- but the convergence results from the aforementioned papers do not give any control in this direction. They are too weak and provide neither an explicit leading order formula for $\si$ nor the next-to-leading order corrections to the asymptotic profile functions. In this paper we derive such formulas and present a refined asymptotic analysis of the high-energy limit for potentials with sufficiently strong singularity.
%
%
% ------------------------------------------------------------------
\subsection{The high-energy limit}
% ------------------------------------------------------------------
%
%
In order to keep the presentation as simple as possible, we restrict our considerations to the example potential
\begin{eqnarray}
\label{Eqn:Pot}
\Phi\at{r} = \frac{1}{m\at{m+1}}\at{\frac{1}{\at{1-r}^m}-m\,r-1}\qquad {\rm with}\quad m\in\oointerval{1}{\infty}\,,
\end{eqnarray}
which is convex and well-defined for $r<1$,
satisfies
\begin{equation*}
\Phi\at{0}=\Phi^\prime\at{0}=0\,,\qquad \Phi^{\prime\prime}\at{0}=1\,,
\end{equation*}
and becomes singular as $r\nearrow1$. The condition $m>1$ is quite essential and shows up several times in our proofs. The other details are less important and
our asymptotic approach can hence be generalized to the case
\begin{equation*}
\eqalign{ %
\mbox{$\Phi$ is convex and smooth on some interval $\ccinterval{a}{b}$ with $\Phi^\prime\at{a}=0$ } \\
\mbox{such that the limit
$\lim_{x\nearrow{b}}\Phi\at{x}\at{b-x}^m$ does exist}\,.
} %
\end{equation*}
This class also includes -- after a reflection with respect to the distance variable
-- all Lennard-Jones-type potential, which blow up on the left
of the global minimum.
\par
To simplify the exposition further, we merely postulate the existence of a family of solitary waves with certain properties but sketch in \sref{sect:justification} how our assumption can be justified rigorously. Specifically, we rely on the following standing assumption, where \emph{unimodal profile} means increasing and decreasing for negative and positive $x$, respectively.
\begin{assumption}[family of high-energy waves]
\label{Ass:Waves}
${\triple{V_\delta}{R_\delta}{\si_\delta}}_{0<\delta<1}$ is a family of solitary waves with the following properties:
\begin{enumerate}
\item
$V_\delta$ and $R_\delta$ belong to $\fspaceL^2\at\Rset\cap\fspace{BC}^1\at\Rset$ and are nonnegative, even, and unimodal.
\item
$V_\delta$ is
normalized by $\norm{V_\delta}_2=1-\delta$ and
$R_\delta$ takes values in $\cointerval{0}{1}$.
\end{enumerate}
Moreover, the potential energy explodes in the sense of
$p_\delta:=\int_\Rset \Phi\bat{R_\delta\at{x}}\dint{x}\to+\infty$ as $\delta\to0$.
\end{assumption}
\begin{figure}[t!]
\centering{%
\includegraphics[width=0.9\textwidth]{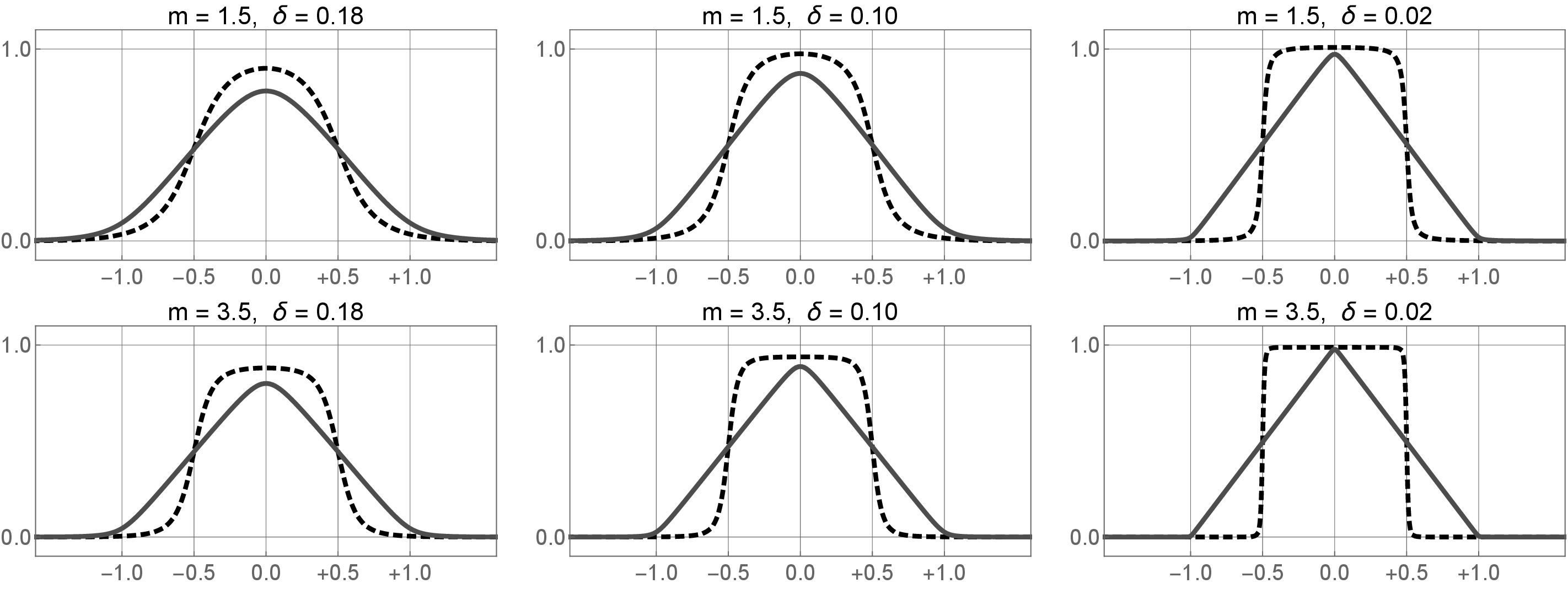}%
}%
\caption{Numerical examples of solitary waves for the potential \eref{Eqn:Pot} and as in Assumption \ref{Ass:Waves}: The graphs of $V_\delta$ (black, dashed) and $R_\delta$ (gray, solid) are plotted for $m=1.5$ (top row) and $m=2.5$ (bottom row). In the high-energy limit $\delta\to0$ (from left to right column), $V_\delta$ and $R_\delta$ approach
the indicator function $V_0$ and the tent map $R_0$, respectively. See \sref{sect:Pre:Waves} for details concerning the numerical scheme.}%
\label{Fig.NumWavesProf}%
\bigskip
\centering{%
\includegraphics[width=0.9\textwidth]{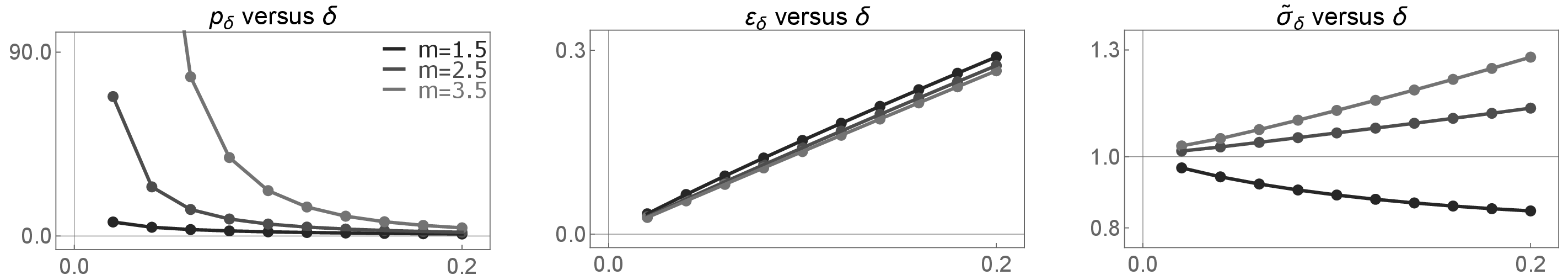}%
}%
\caption{Parameter plots for three different choices of $m$ and the simulations from Figure \ref{Fig.NumWavesProf}: $p_\delta$ represents the potential energy, $\eps_\delta$ is the amplitude parameter from \eref{Eqn:DefEpsMu}, and
$\tilde{\si}_\delta := \si_\delta/\at{\ol{\mu}^2 \eps_\delta^m}$ measures
the relative deviation of the speed parameter $\si$ with respect to the
asymptotic value from \eref{Cor:TS.ConvPrms.Eqn1}}%
\label{Fig.NumWavesData}%
\end{figure}
Beside of $\delta$ there exist two other small quantities, namely
\begin{equation}
\label{Eqn:DefEpsMu}
\eps_\delta := 1-R_\delta\at{0}\,,\qquad
\mu_\delta:=\sqrt{\si_\delta\,\eps_\delta^{m+2}},
\end{equation}
which feature prominently in the asymptotic analysis. The amplitude parameter $\eps_\delta$ quantifies the impact of the singularity and appears naturally
in many of the estimates derived below. The parameter $\mu_\delta$, which
looks rather artificial at a first glance, is also very important as
it determines the length scale for the leading order corrections to the asymptotic
profile functions $V_0$ and $R_0$.
\par
For the interaction potential \eref{Eqn:Pot} and the waves from Assumption \ref{Ass:Waves}, the existing results for the limit $\delta\to0$ are illustrated in Figures \ref{Fig.NumWavesProf}, \ref{Fig.NumWavesData} and
can be summarized as follows.
\begin{theorem}[localization theorem]
\label{Thm:Localization}
In the high-energy limit, we have
\begin{equation*}
\norm{V_\delta-V_0}_2+\norm{R_\delta-R_0}_\infty +\eps_\delta+\mu_\delta\quad\xrightarrow{\;\;\delta\to0\;\;}\quad0
\end{equation*}
with
\begin{equation*}
V_0\at{x}:=\chi\at{x}\,\qquad R_0\at{x}:=\max\big\{0,\,1-\abs{x}\big\}\,,
\end{equation*}
where $\chi$ denotes the indicator function of the interval $\ccinterval{-\frac12}{+\frac12}$.
\end{theorem}
We give a short proof in \sref{sect:preliminaries}. The corresponding results in \cite{FM02,Tre04} also provide lower and upper bounds but no explicit expansion for $\si_\delta$.
%
%
% ------------------------------------------------------------------
\subsection{Statement of the asymptotic result}
% ------------------------------------------------------------------
%
%\begin{mhchange}
%
Our strategy for deriving a refined asymptotic analysis is to
blow up the profile functions near the critical spatial positions and to identify
equations that determine the asymptotic wave shape with respect to a rescaled space variable $\tilde{x}$. Specifically, we use the
\emph{transition scaling} in order to
describe the asymptotic velocity profile near $x=\pm\frac12$, while the
distance profile can be rescaled at both the \emph{tip position} $x=0$ and the \emph{foot positions} $x=\pm1$, see Figure \ref{Fig.CartoonScaling} for an illustration.
Our main findings can informally be summarized as follows.%
\begin{result*}
In the high-energy limit $\delta\to0$, all relevant information on $\triple{R_\delta}{V_\delta}{\si_\delta}$
can be obtained from the function $\tilde{S}_0$, which is defined by
the ODE initial value problem
\begin{equation}
\label{Eqn:LimitIVP}
\tilde{S}_0^{\prime\prime}\at{\tilde{x}}=
\frac{2}{m+1}\cdot\frac{1}{\at{1+\tilde{S}_0\at{\tilde{x}}}^{m+1}}\,,\qquad \tilde{S}_0\at{0}=\tilde{S}_0^\prime\at{0}=0
\end{equation}
and plotted in Figure \ref{Fig.CartoonLimit}. More precisely,
\begin{enumerate}
\item
the velocity profile $V_\delta$ converges under the \emph{transition scaling},
\item the distance profile $R_\delta$ converges under both the
\emph{tip scaling} and the \emph{foot scaling},
\item the rescaled parameters $\delta^m\si_\delta$, $\delta^{-1}\eps_\delta$, and $\delta^{-1}\mu_\delta$ converge,
\end{enumerate}
where the respective limit objects can be expressed in terms of $\tilde{S}_0$ and all error terms are at most of order $\DO{\delta^m}$.
\end{result*}
\begin{figure}[t!]
\centering{%
\includegraphics[width=0.6\textwidth]{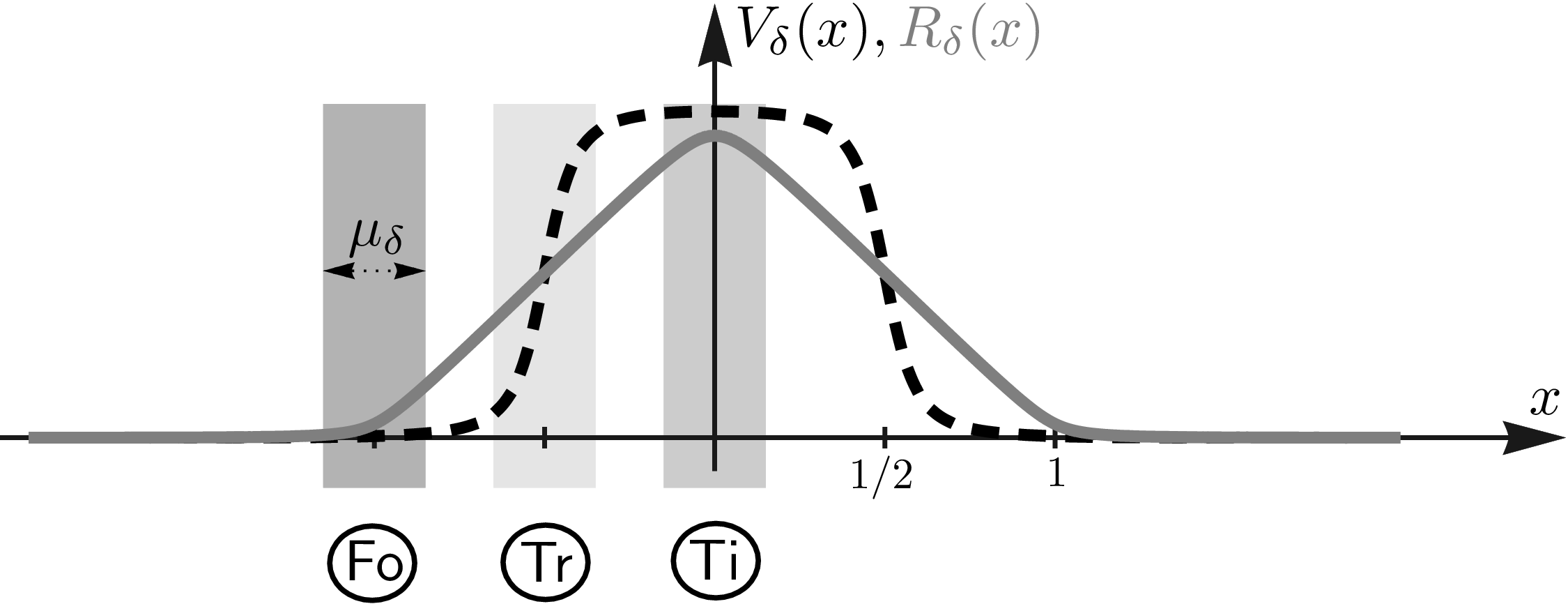}%
}%
\caption{Schematic representation of the different scalings: The transition scaling describes the jump-like behavior of $V_\delta$ near $x=\pm\mbox{$\frac12$}$ while the foot and the tip  scaling magnify the turns of $R_\delta$ at
$x\approx\pm1$ and $x\approx0$, respectively. The width parameter $\mu_\delta$ is introduced in \eref{Eqn:DefEpsMu} and satisfies $\mu_\delta\sim\eps_\delta$ according to Corollary \ref{Cor:ScalingLaws}.}%
\label{Fig.CartoonScaling}%
\bigskip
\centering{%
\includegraphics[width=0.85\textwidth]{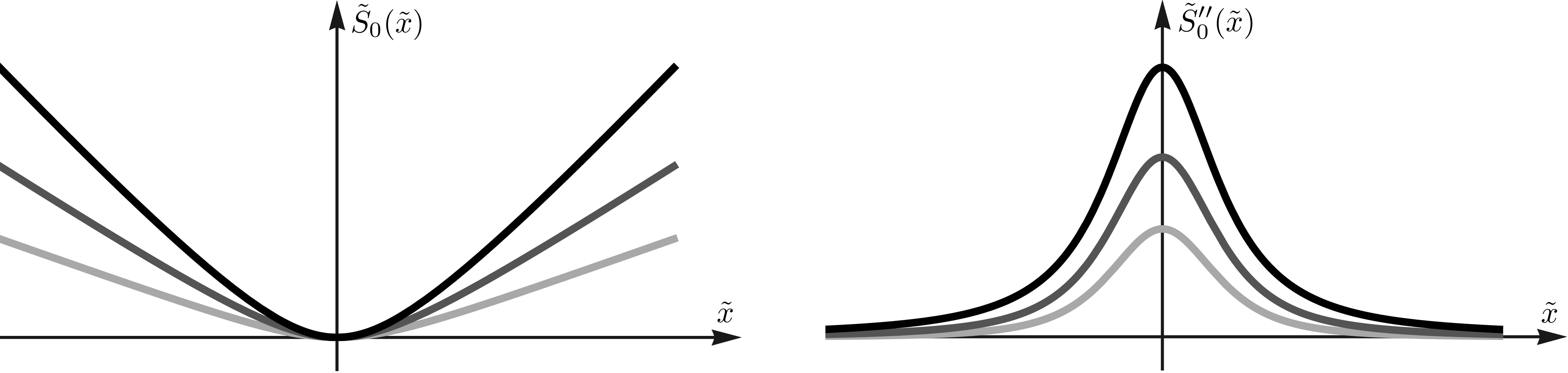}
}%
\caption{Graph of the function $\tilde{S}_0$ and its second derivative for $m=m_1$ (black) and $m=m_2$ (dark gray)  and $m=m_3$ (light gray) with $m_1<m_2<m_3$.}
\label{Fig.CartoonLimit}%
\end{figure}
The details concerning the convergence under the tip, the transition, and the foot scaling are presented in the Theorems \ref{Thm:TS.Asymptotics}, \ref{Thm:TrS.MR}, and \ref{Thm:FoS.MR}, respectively, while Corollary \ref{Cor:ScalingLaws} provides the explicit scaling laws for $\eps_\delta$, $\mu_\delta$, and $\si_\delta$. Moreover, the combination of all partial estimates gives rise to the global approximation results in
Theorem~\ref{Thm:GlobApp} and Corollary~\ref{Cor:GlobApp}.
\par
The above results provide an improved understanding of the high-energy limit of solitary waves. In particular, it seems that our asymptotic formulas can be used to control the spectrum of the linearized traveling wave equation, see the brief and preliminary discussion in \sref{sect:eigenproblem}.
\par
The paper is organized as follows. In the remainder of the introduction we prove
the localization theorem and discuss both the results from \cite{FM02,Tre04} and the justification of Assumption \ref{Ass:Waves} in greater detail.
Then \sref{sect:tipscaling} is devoted to the
tip scaling, which turns out to be most fundamental step in our asymptotic analysis. In particular, we identify the intrinsic scaling parameters in \sref{sect:tipscaling.1} and link afterwards
in \sref{sect:tipscaling.2} the rescaled distance profile to the initial value problem \eref{Eqn:LimitIVP}. In \sref{sect:further} we finally employ the results on the tip scaling and establish all other asymptotic formulas.
%
%
%
%
% ------------------------------------------------------------------
\subsection{Preliminaries}\label{sect:preliminaries}
% ------------------------------------------------------------------
%
%
%\begin{mhchange}
%
In this section we prove Theorem \ref{Thm:Localization} since it provides the starting point for our asymptotic analysis in \sref{sect:tipscaling} and \sref{sect:further}. To this end it is convenient to reformulate the advance-delay-differential equation \eref{Eqn:TW.Diff} as
\begin{equation}
\label{Eqn:TW.Int}
 R = AV\,,\qquad  \si\, V =A \Phi^\prime\at{R}\,,
\end{equation}
where the operator $A$ stands for the convolution with the indicator function $\chi$. This reads
\begin{equation*}
\bat{AV}\at{x}=\int_{x-\mbox{$\frac12$}}^{x+\mbox{$\frac12$}} V\at{y}\dint{y}
\end{equation*}
and the elimination of $R$ reveals that \eref{Eqn:TW.Diff} can be viewed as a symmetric but nonlinear and nonlocal eigenvalue problem for the eigenvalue $\si$ and the eigenfunction $V$. The proof that \eref{Eqn:TW.Int} implies \eref{Eqn:TW.Diff} is straight forward and involves only differentiation with respect to $x$; for the reversed statement one has to eliminate the constants of integration by the decay condition $V\in\fspaceL^2\at\Rset$.
\par
Using elementary analysis such as H\"older's inequality we readily verify the estimates
\begin{equation}
\label{Eqn:EstimatesForA}
\norm{AV}_2\leq \norm{V}_2\,,\qquad
\norm{AV}_\infty\leq \norm{V}_2\,,\qquad \norm{\at{AV}^\prime}_2\leq2\, \norm{V}_2\,,
\end{equation}
for any function $V\in\fspaceL^2\at\Rset$, and this implies that the potential energy
\begin{equation}
\label{Eqn:DefEnergy}
\calP\at{V}:=\int_{\Rset}\Phi\bat{\at{AV}\at{x}}\dint{x}\,,
\end{equation}
is well defined as long as $\norm{V}_2<1$. Moreover, we get
\begin{equation}
\label{Eqn:EstimatesForDelta}
\norm{R_\delta}_\infty=R_\delta\at{0}=1-\eps_\delta\leq\norm{V_\delta}_2=1-\delta
\end{equation}
for the family from Assumption \ref{Ass:Waves}.
\begin{lemma}[variant of the localization theorem]
\label{Lem:SimpleConv}
The estimates
\begin{equation}
\label{Lem:SimpleConv.Eqn1}
\norm{V_\delta - \chi}_2\leq C\eps_\delta,\qquad  \norm{R_\delta- A\chi}_\infty\leq C\eps_\delta
\end{equation}
and
\begin{equation}
\label{Lem:SimpleConv.Eqn2}
 c \, \eps_\delta \leq \mu_\delta\leq C\sqrt{\eps_\delta}
\end{equation}
hold for some constants $c$, $C$ independent of $\delta$. Moreover, we have $\eps_\delta\to0$
as $\delta\to0$.
\end{lemma}
\begin{proof}
We start with the identities
\begin{equation}
\label{Lem:SimpleConv.PEqn1a}
R_\delta\at{0}=\bskp{V_\delta}{\chi}\,,\qquad
\norm{V_\delta-\chi}_2^2=\norm{V_\delta}_2^2 +\norm{\chi}_2^2 - 2 \bskp{V_\delta}{\chi}\,,
\end{equation}
where $\skp{\cdot}{\cdot}$ denotes the usual inner product in $\fspaceL^2\at\Rset$, and
observe that \eref{Eqn:TW.Int} implies
\begin{equation}
\label{Lem:SimpleConv.PEqn1b} \si_\delta\at{1-\delta}^2=\bskp{\Phi^\prime\at{R_\delta}}{R_\delta}\,.
\end{equation}
Since \eref{Eqn:EstimatesForDelta} yields $\delta\leq\eps_\delta$, we find
\begin{equation*}
0\leq \norm{V_\delta-\chi}_2^2=\at{1-\delta}^2+1-2\at{1-\eps_\delta}\leq C\eps_\delta
\end{equation*}
and hence \eref{Lem:SimpleConv.Eqn1}$_1$, which in turn implies
\eref{Lem:SimpleConv.Eqn1}$_2$ thanks to  $R_\delta- A\chi=A(V_\delta -\chi)$ and \eref{Eqn:EstimatesForA}$_2$. By \eref{Eqn:TW.Int} we also have
\begin{equation*}
\mu_\delta^2\, V_\delta\at{0}=\eps_\delta^{m+2}\int_{-\frac12}^{+\frac12}\Phi^\prime\bat{R_\delta\at{x}}\leq \eps_\delta^{m+2}\Phi^\prime\at{1-\eps_\delta}\leq C \eps_\delta\,,
\end{equation*}
and this provides the upper bound in \eref{Lem:SimpleConv.Eqn2} since the unimodality of $V_\delta$ combined with \eref{Lem:SimpleConv.Eqn1}$_1$ guarantees that $\liminf_{\delta\to0} V_\delta\at{0}>0$. To obtain the corresponding lower bound, we notice that \eref{Eqn:TW.Diff} ensures
\begin{equation*}
\norm{R_\delta^{\prime\prime}}_\infty\leq \frac{4\norm{\Phi^\prime\at{R_\delta}}_\infty}{\si_\delta} \leq \frac{C}{\si_\delta\eps_\delta^{m+1}}
\end{equation*}
and hence
\begin{equation*}
R_\delta\at{x}\geq 1-C\eps_\delta\quad{\rm for}\;{\rm all}\quad \abs{x}<\sqrt{\si_\delta \eps_\delta^{m+2}}=\mu_\delta
\end{equation*}
due to $R_\delta^\prime\at{0}=0$ and $R_\delta\at{0}=1-\eps_\delta$. Combining this with \eref{Lem:SimpleConv.PEqn1b} we obtain
\begin{equation*}
\frac{\mu_\delta^2}{\eps_{\delta}^{m+2}}\at{1-\delta}^2\geq\int_{-\mu_\delta}^{+\mu_\delta} \Phi^\prime\bat{R_\delta\at{x}}R_\delta\at{x}\dint{x}\geq  c \frac{
\mu_\delta}{\eps_\delta^{m+1}},
\end{equation*}
and the proof of \eref{Lem:SimpleConv.Eqn2} is complete. Finally, the properties of $\Phi$ imply
\begin{equation*}
p_\delta=\calP\at{V_\delta}\leq \eps_\delta^{-m}\norm{R_\delta}_2^2\leq C \eps_{\delta}^{-m}
\end{equation*}
so $\eps_\delta\to0$ is a consequence of $p_\delta\to\infty$.
\end{proof}
%
%
% ------------------------------------------------------------------
\subsection{Justification of Assumption \ref{Ass:Waves}\label{sect:Pre:Waves}}\label{sect:justification}
% ------------------------------------------------------------------
%
We briefly sketch how Assumption \ref{Ass:Waves} can be justified using a constrained optimization approach. All key arguments are presented in \cite{Her10} for non-singular
potentials $\Phi$ but can easily be adapted to the potential \eref{Eqn:Pot}. At the end of this section we also discuss the results from \cite{Tre04} and \cite{FM02}.
\par
The variational approach from \cite{Her10} is based on the potential energy functional
\eref{Eqn:DefEnergy}, which is convex and G\^{a}teaux differentiable on the open
unit ball in $\fspaceL^2\at{\Rset}$; the derivative of $\calP$ is given by
\begin{equation*}
\partial_V \calP\at{V} = A \Phi^\prime\at{AV}\,,
\end{equation*}
so the traveling wave equation \eref{Eqn:TW.Int} is equivalent to $\si V=\partial_V \calP\at{V}$.
We further introduce the cone $\calC$ of all $\fspaceL^2$-functions that are even, unimodal and nonnegative, i.e. we set
\begin{equation*}
\calC:=\{V\in\fspaceL^2\at\Rset:\mbox{$0\leq V\at{x}\leq V\at{y}=V\at{-y}$ for almost all  $x\leq y\leq 0$} \}\,.
\end{equation*}
The key observation is that solitary waves as in Assumption \ref{Ass:Waves} can be constructed as solutions to the constrained optimization problem
\begin{eqnarray}
\label{Eqn:OptProb}
\eqalign{
\mbox{Maximize $\calP$ under the norm constraint $\norm{V}_2=1{-}\delta$ and} \\
\mbox{the shape constraint $V\in\calC$.}
}
\end{eqnarray}
In the existence proof one has to ensure that maximizers do in fact
exist and that the shape constraint does not contribute to the Euler-Lagrange equation for the maximizer. With respect to the latter issue we introduce
the \emph{improvement operator}
\begin{equation*}
\calT_\delta\at{V}:=\at{1-\delta}\frac{A\Phi^\prime\at{AV}}{\norm{A\Phi^\prime\at{AV}}_2}\,.
\end{equation*}
We are now able to describe the key arguments in the variational existence proof for solitary waves with profiles in $\calC$.
\begin{lemma}[three ingredients]
\label{Lem:OptProb}
\quad
\begin{enumerate}
\item Since $\Phi$ is strictly super-quadratic, each maximizing sequence for \eref{Eqn:OptProb} is strongly compact.
\item The cone $\calC$ is invariant under the actions of both
the convolution operator $A$ and the superposition operator $\Phi^\prime$. In particular, $V\in\calC$ implies $R\in\calC$ and
$\calT_\delta\at{V}\in\calC$.
\item We have $\calP\bat{\calT_\delta\at{V}}\geq\calP\at{V}$, where the equality sign holds if an only if $V$ is a fixed point of $\calT_\delta$.
\end{enumerate}
\end{lemma}
\begin{proof}[Sketch of the proof]
The main steps can be summarized as follows:
\begin{enumerate}
\item
The assertion follows by a variant of the \emph{Concentration Compactness Principle}.
\item
The invariance properties can be checked by straight forward calculations.
\item
Since $\Phi$ is convex, we have
\begin{eqnarray*}
\calP\bat{\calT_\delta\at{V}}-\calP\at{V}&\geq&\bskp{\partial_V \calP\at{V}}{\calT_\delta\at{V}-V}\\&=&\si\at{V}\bskp{\calT_\delta\at{V}}{\calT_\delta\at{V}-V}
\\&=&
\mbox{$\frac12$}{\si\at{V}\norm{\calT_\delta\at{V}-V}_2^2}\,,
\end{eqnarray*}
where we used $\norm{\calT_\delta\at{V}}_2=\norm{V}_2=1-\delta$
and that $\si\at{V}:=\norm{A\Phi^\prime\at{AV}}_2/\at{1-\delta}$ is well defined as long as $\calP\at{V}>0$.
\end{enumerate}
The details can be found in \cite{Her10}.
\end{proof}
\begin{corollary}[variational existence proof]
For any $0<\delta<1$, there exists a solitary wave with $V_\delta, R_\delta\in
\calC\cap\fspace{BC}^\infty\at\Rset$ and $\norm{V_\delta}=1-\delta$. Moreover, we have
$\calP\at{V_\delta}\to\infty$ as $\delta\to0$.
\end{corollary}
\begin{proof}
The existence of a solution $V_\delta$ to \eref{Eqn:OptProb} can be established by the \emph{Direct Method} due to the compactness result from Lemma \ref{Lem:OptProb} and since
$\calP$ is strongly continuous. Moreover, any maximizer $V_\delta$ satisfies
\begin{equation*}
\calP\at{V_\delta}\geq \calP\bat{\calT_\delta \at{V_\delta}}\,,\qquad
\calP\at{V_\delta}\geq \calP\bat{\at{1-\delta}\,\chi}\,.
\end{equation*}
The first estimate implies
$\calP\at{V_\delta}= \calP\bat{T_\delta \at{V_\delta}}$, so
$V_\delta$ satisfies the traveling wave equation \eref{Eqn:TW.Int} and is therefore smooth. We also compute
\begin{equation*}
\calP\bat{\at{1-\delta}\chi}=2\int_{0}^1\Phi\bat{\at{1-\delta}\at{1-x}}\dint{x}\quad\xrightarrow{\delta\to0}\quad +\infty
\end{equation*}
and the proof is complete.
\end{proof}
The improvement operator $\calT_\delta$ can also be used to compute solitary waves numerically. In fact, imposing homogeneous Dirichlet boundary conditions on a large but bounded and fine grid, the integral operator $A$ can easily be discretized by Riemann sums. The resulting recursive scheme exhibits very good convergence properties; it has been applied to a wide range of potentials, see for instance \cite{EP05, Her10}, and also been used to compute the numerical data displayed in Figures \ref{Fig.NumWavesProf} and \ref{Fig.NumWavesData}.
\bigpar
A different variational framework has been introduced in \cite{FW96} and later been applied to the high-energy limit in \cite{FM02}. The
key idea there is to minimize the kinetic energy term $\frac12\norm{V}_2^2$ subject to a prescribed value of $p=\calP\at{V}$. The results from \cite{FM02} imply for the potential \eref{Eqn:Pot} that solitary waves converge as $p\to\infty$ to the limit function $\chi$ and satisfy Assumption \ref{Ass:Waves}
(though, strictly speaking, neither the unimodality nor the evenness of the profile functions $R$ and $V$ have been shown).
In this context we emphasize again that uniqueness of Hamiltonian lattice waves
is a notoriously difficult and an almost completely open problem. It is commonly believed that all variational and non-variational approaches provide -- up to reparametrizations and for, say, convex potentials -- the same family of solitary wave but there seems to be no proof so far.
\par
A non-variational existence proof for solitary waves with high energy has been given in \cite{Tre04} using a carefully designed fixed-point argument for the (negative) distance profile $R$ in the space of exponentially decaying functions. In our notations, the
smallness parameter is $\eps$ and the waves are shown to satisfy
$\norm{R-A\chi}=\DO{\eps}$ in some suitably chosen norm. We therefore expect that the waves constructed in \cite{Tre04} also satisfy Assumption \ref{Ass:Waves}, although the justification of the unimodality might be an issue.
%
%
%
%
% ------------------------------------------------------------------
\section{Main result on the tip scaling \texorpdfstring{of $R_\delta$}{}}
\label{sect:tipscaling}
% ------------------------------------------------------------------
%
Our first goal is to describe the asymptotic behavior of the distance profile $R_\delta$ near $x=0$ by showing that it converges as $\delta\to0$ under an appropriately defined rescaling to some nontrivial limit function. In view of theorem \ref{Thm:Localization} and the numerical simulations from figure \ref{Fig.NumWavesData} we expect that both the variable $x$ as well as the shifted amplitude variable $1-R_\delta\at{\delta}$ must be scaled with certain powers
of $\delta$. A naive ansatz, however, does not work here because we lack
a priori scaling relations between the small parameters
$\delta$, $\eps_\delta$, and $\mu_\delta$. For instance, if we would start with the rescaling
\begin{equation*}
R_\delta\at{x}=1-{\delta}^{\ga_1} \bar{R}_\delta\at{\delta^{\ga_2}x}\,,
\end{equation*}%
we could not eliminate $\si_\delta$ in the leading order equation.
To overcome this problem we base our analysis on an implicit scaling, which magnifies the amplitude with $\eps_\delta$ but defines the rescaled space variable
by
\begin{equation*}
\tilde{x}=\mu_\delta\, x\,.
\end{equation*}
In this way we obtain an explicit leading order equation that does not involve any unknown parameter and can hence be solved.  Moreover, the corresponding solution finally allows us to identify the scaling relations between the different parameters; at the end it turns out that $\delta$, $\eps_\delta$ and $\mu_\delta$ are all proportional to each other, see Corollary \ref{Cor:ScalingLaws}.
%
%
%
% ------------------------------------------------------------------
\subsection{Implicit rescaling \texorpdfstring{of $R_\delta$}{}}
\label{sect:tipscaling.1}
% ------------------------------------------------------------------
%
%
\begin{figure}[t!]
\centering{%
\includegraphics[width=0.95\textwidth]{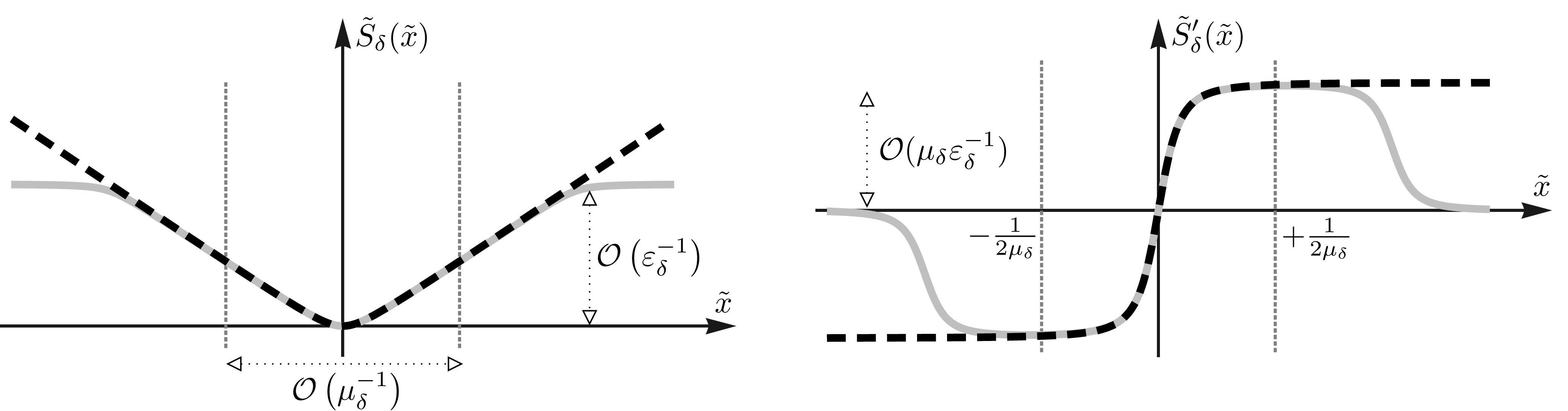}
}%
\caption{Cartoon of the tip scaling: The function $\tilde{S}_\delta$ and its derivative for $\delta>0$ (gray, solid) and $\delta=0$ (black, dashed). The dotted vertical lines enclose the symmetric interval  $J_\delta$ which has length $1/\mu_\delta\gg1$. The convergence $\tilde{S}_\delta\to\tilde{S}_0$ as $\delta\to0$ implies $\mu_\delta\sim\eps_\delta$, see Theorem \ref{Thm:TS.Asymptotics} and Corollary \ref{Cor:TS.ConvPrms}.}%
\label{Fig.Convergence.TipScaling}%
\end{figure}
In order to derive asymptotic formulas for $R_\delta$ near $x=0$, we
define the rescaled distance profile
\begin{eqnarray}
\label{Eqn:TipS.Def1}%
\tilde{S}_\delta\at{\tilde{x}}:=
\frac{R_\delta\at{0}-R_\delta\at{\mu_\delta\tilde{x}}}{\eps_\delta}=
\frac{1-\eps_\delta-R_\delta\at{\mu_\delta\tilde{x}}}{\eps_\delta}
\end{eqnarray}
and obtain an even function which satisfies
\begin{equation}
\label{Eqn:TipS.Id0}%
\tilde{S}_\delta\at{0}=\tilde{S}_\delta^\prime\at{0}=0\,,
\qquad \tilde{S}_\delta\at{\tilde{x}}=\tilde{S}_\delta\at{-\tilde{x}}\,,
\end{equation}
 see Figure \ref{Fig.Convergence.TipScaling} for an illustration. We also introduce the auxiliary functions
\begin{eqnarray}
\label{Eqn:TipS.Def2}%
\tilde{F}_\delta\at{\tilde{x}}&:=\eps_\delta^{m+1}\Phi^\prime\Bat{R_\delta\at{\mu_\delta\tilde{x}}}
\\%
\label{Eqn:TipS.Def3}%
\tilde{G}_\delta\at{\tilde{x}}&:=\eps_\delta^{m+1}\Phi^\prime\Bat{R_\delta\at{-1+\mu_\delta\tilde{x}}},
\end{eqnarray}
as well as the intervals
\begin{equation*}
I_\delta:=\left[0,\,\frac{1}{2\mu_\delta}\right]\,,\qquad J_\delta = \at{-I_\delta}\cup I_\delta
\end{equation*}
and study the limit of $\tilde{S}_\delta$ restricted to $J_\delta$.
\bigpar
Employing the identity \eref{Eqn:TW.Diff}$_1$ as well as \eref{Eqn:DefEpsMu}
we readily verify
\begin{eqnarray*}
\tilde{S}^{\prime\prime}_\delta\at{\tilde{x}}&=&
\frac{\mu_\delta^2}{\eps_\delta}\Bat{V_\delta^\prime\at{\mu_\delta\tilde{x}-\mbox{$\frac12$}}-V_\delta^\prime\at{\mu_\delta\tilde{x}+\mbox{$\frac12$}}}
\\&=&\si_\delta\eps_\delta^{m+1}\Bat{V_\delta^\prime\at{\mu_\delta\tilde{x}-\mbox{$\frac12$}}-V_\delta^\prime\at{\mu_\delta\tilde{x}+\mbox{$\frac12$}}}
\end{eqnarray*}
and by \eref{Eqn:TW.Diff}$_2$ we arrive at
\begin{equation}
\label{Eqn:TipS.Id1}%
\tilde{S}^{\prime\prime}_\delta\at{\tilde{x}}= 2\tilde{F}_\delta\at{\tilde{x}}-\tilde{G}_\delta\at{\tilde{x}}-\tilde{G}_\delta\at{-\tilde{x}}\,,
\end{equation}
where we used that $R_\delta\at{1+\mu_\delta\tilde{x}}= R_\delta\at{-1-\mu_\delta\tilde{x}}$.
The definition of $\tilde{G}_\delta$ combined with the unimodality of $R_\delta$ implies
\begin{equation}
\label{Eqn:TS.GEstimates}
0\leq\tilde{G}_\delta\at{\tilde{x}}\leq
\tilde{G}_\delta\at{\frac{1}{2\mu_\delta}}=
\eps_\delta^{m+1}\Phi^\prime\bat{R_\delta\at{\mbox{$\frac12$}}}\leq C\eps_\delta^{m+1}\mbox{\quad for all $\tilde{x}\in J_\delta$ }
\end{equation}
and from \eref{Eqn:TipS.Def2} we get
\begin{equation}
\label{Eqn:TipS.Id2}
\tilde{F}_\delta\at{\tilde{x}}=
\frac{\Psi\at{\eps_\delta+\eps_\delta\tilde{S}_\delta\at{\tilde{x}}}}{\at{1+\tilde{S}_\delta\at{\tilde{x}}}^{m+1}}\,,
\end{equation}
where the function $\Psi:\ccinterval{0}{1}\to\Rset_+$ with
\begin{equation*}
\Psi\at{s}:=\Phi^\prime\at{1-s}s^{m+1}\quad \mbox{for}\quad 0<s\leq1\,,\qquad
\Psi\at{0}:=\lim_{s\searrow0}\Psi\at{s}=\frac{1}{m+1}
\end{equation*}
is smooth and positive. In particular, on the interval $J_\delta$ we find
\begin{equation*}
\tilde{S}_\delta^{\prime\prime} \at{\tilde{x}}\approx
2\tilde{F}_\delta \at{\tilde{x}}\approx
\frac{2}{m+1}\cdot\frac{1}{\at{1+\tilde{S}_\delta\at{\tilde{x}}}^{m+1}}\,,
\end{equation*}
and conclude that $\tilde{S}_\delta$ satisfies the initial value problem \eref{Eqn:LimitIVP} from the introduction up to small error terms.
\begin{lemma}[solution of the limit problem]
\label{Lem:TS.LimitProfile}
The initial value problem \eref{Eqn:LimitIVP} has a unique solution which is even, nonnegative, and convex. This solution $\tilde{S}_0$
grows linearly for $\tilde{x}\to\pm\infty$ as it
satisfies
\begin{equation}
\label{Lem:TS.LimitProfile.Eqn2a}
\abs{\tilde{S}_0^\prime\at{\tilde{x}}-
\ol{\mu}\,\mathrm{sgn}\at{\tilde{x}}}\leq \frac{C}{\at{1+\tilde{x}}^{m}}
\end{equation}
and
\begin{equation}
\label{Lem:TS.LimitProfile.Eqn2b}
\abs{\tilde{x}\tilde{S}_0^\prime\at{\tilde{x}}-
\tilde{S}_0\at{\tilde{x}}-\ol{\ka}}\leq\frac{C}{\at{1+\tilde{x}}^{m-1}}
\end{equation}
for all $\tilde{x}\in\Rset$ with
\begin{equation}
\label{Lem:TS.LimitProfile.Consts}
\ol{\mu}:=\frac{2}{\sqrt{m\at{m+1}}}\,,\quad \ol{\ka}:=\int_{0}^\infty \tilde{x}\,\tilde{S}_0^{\prime\prime}\at{\tilde{x}}\dint{\tilde{x}}
\,,\quad
\ol{\eta}:=\int_{0}^\infty {\tilde{S}_0\at{\tilde{x}}}\,\tilde{S}_0^{\prime\prime}\at{\tilde{x}}\dint{\tilde{x}}
\end{equation}
and some constant $C$ which depends only on $m$.
\end{lemma}
\begin{proof}
The planar and autonomous Hamiltonian
ODE \eref{Eqn:LimitIVP}$_1$  admits the conserved quantity
\begin{equation*}
E_{\rm tot}\at{\tilde{x}}:=\mbox{$\frac12$}\Bat{\tilde{S}_0^\prime\at{\tilde{x}}}^2+
E_{\rm pot}\at{\tilde{x}}\,,\qquad E_{\rm pot}\at{\tilde{x}}:=
\frac{\mbox{$\frac12$}\ol{\mu}^2}{\at{1+\tilde{S}_0\at{\tilde{x}}}^{m}}
\end{equation*}
with value $E_{\rm tot}\at{\tilde{x}}=E_{\rm tot}\at{0}=\frac12\ol{\mu}^2$ for all $\tilde{x}$. A simple phase plane analysis reveals that $\tilde{S}$ is even and that
both $\tilde{S}_0$ and $\tilde{S}_0^\prime$ are strictly increasing for $\tilde{x}>0$,
see Figure \ref{Fig.CartoonLimit} for an illustration. In particular, we have
\begin{equation*}
\tilde{S}_0^\prime\at{\tilde{x}}\quad\xrightarrow{\tilde{x}\to\infty}\quad  \sqrt{2 E_{\rm tot}\at{0}}=\ol{\mu}>0
\end{equation*}
so the conservation law implies
\begin{eqnarray*}
\abs{\tilde{S}_0^\prime\at{\tilde{x}}-\sqrt{2 E_{\rm tot}\at{0}}}&=&\abs{\sqrt{2E_{\rm tot}\at{0}-2E_{\rm pot}\at{\tilde{x}}}-\sqrt{2 E_{\rm tot}\at{0}}}\\&\leq& C\,E_{\rm pot}\at{\tilde{x}}\leq \frac{C}{\at{1+\tilde{x}}^m}\,,
\end{eqnarray*}
and hence \eref{Lem:TS.LimitProfile.Eqn2a}. Moreover, the even function  $\tilde{K}_0$ with $\tilde{K}_0\at{\tilde{x}}:=\tilde{x}\tilde{S}_0^\prime\at{\tilde{x}}-
\tilde{S}_0\at{\tilde{x}}$ satisfies
\begin{equation*}
0\leq\tilde{K}_0^\prime\at{\tilde{x}}=\tilde{x}\tilde{S}_0^{\prime\prime}\at{\tilde{x}}\leq \frac{C}{\at{1+\tilde{x}}^m}\qquad \mbox{for all}\quad \tilde{x}>0\,,
\end{equation*}
so  $\tilde{K}_0^\prime$ is integrable due to $m>1$. The constant $\ol{\ka}=\lim_{\tilde{x}\to\infty}\tilde{K}_0\at{\tilde{x}}$ is therefore
well-defined and \eref{Lem:TS.LimitProfile.Eqn2b} follows immediately from the estimate for $\tilde{K}_0^\prime\at{\tilde{x}}$. Finally, $\ol{\eta}$ is well-defined since
the integrand is continuous and decays as $\tilde{x}^{-m}$ for  $\tilde{x}\to\infty$.
\end{proof}
\begin{figure}[t!]
\centering{%
\includegraphics[width=0.3\textwidth]{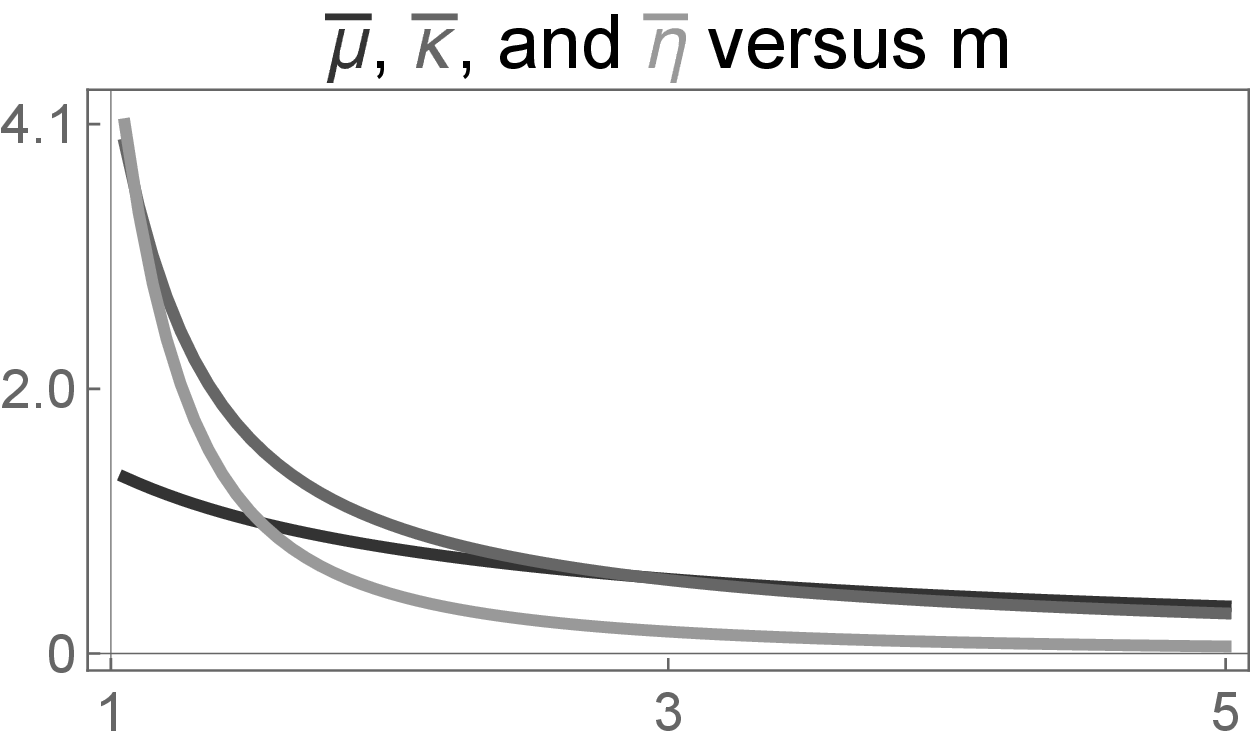}
}%
\caption{Numerical values for the constants $\ol{\mu}$, $\ol{\ka}$, and $\ol{\eta}$ from \eref{Lem:TS.LimitProfile.Consts}, which provide the leading and the next-to-leading order terms in the scaling relations between $\eps_\delta$, $\mu_\delta$, and $\si_\delta$, see Corollary \eref{Cor:ScalingLaws}.
} %
\label{Fig.Coefficients}%
\end{figure}
There seems to be no simple way to compute the constants $\ol{\ka}$ and $\ol{\eta}$ as functions of $m$
but numerical values are presented in Figure \ref{Fig.Coefficients}.
%
%
%
% ------------------------------------------------------------------
\subsection{Asymptotic formulas \texorpdfstring{for $\tilde{S}_\delta$}{}}
\label{sect:tipscaling.2}
% ------------------------------------------------------------------
%
%
%
We now able to formulate and prove our main asymptotic result.
\begin{theorem}[asymptotics of $\tilde{S}_\delta$]
\label{Thm:TS.Asymptotics}
Any function $\tilde{S}_\delta$ is strictly increasing and convex on $I_\delta$. Moreover, the estimates
\begin{equation}
\label{Thm:TS.Asymptotics.Eqn1}
\sup\limits_{\tilde{x}\in{J_\delta}} \abs{\tilde{S}_\delta\at{\tilde{x}}-
\tilde{S}_0\at{\tilde{x}}}\leq C\eps_\delta^{m-1}
\end{equation}
and
\begin{equation}
\label{Thm:TS.Asymptotics.Eqn2}
\sup\limits_{\tilde{x}\in{J_\delta}} \abs{\tilde{S}_\delta^\prime\at{\tilde{x}}-
\tilde{S}_0^\prime\at{\tilde{x}}}\leq C\eps_\delta^m\,,\qquad
\sup\limits_{\tilde{x}\in{J_\delta}} \abs{\tilde{S}_\delta^{\prime\prime}\at{\tilde{x}}-
\tilde{S}_0^{\prime\prime}\at{\tilde{x}}}\leq C\eps_\delta^{m+1}
\end{equation}
hold for all $0<\delta<1$ and a constant $C$ independent of $\delta$.
\end{theorem}
\begin{proof} Since $\tilde{S}_\delta$ and $\tilde{S}_0$ are even functions it suffices to consider
$\tilde{x}\in I_\delta$.
\par \ul{\emph{Dynamics of $\tilde{S}_\delta$}}: %
By \eref{Eqn:TipS.Id2} and due to
$\Psi\at{0}-\Psi\at{s}=Cs^{m+1}$ we have
\begin{equation*}
\abs{\tilde{F}_\delta\at{\tilde{x}}-
\frac{\Psi\at{0}}{\at{1+\tilde{S}_\delta\at{\tilde{x}}}^{m+1}}}\leq C\eps_\delta^{m+1}\,,
\end{equation*}
and in view of \eref{Eqn:TipS.Id2} and \eref{Eqn:TS.GEstimates} we conclude that $\tilde{S}_\delta$ satisfies on the interval $I_\delta$ the ODE
\begin{equation}
\label{Thm:TS.Asymptotics.PEqn0}
\tilde{S}_\delta^{\prime\prime}\at{\tilde{x}}=\frac{\Psi\at{0}}{\at{1+\tilde{S}_\delta\at{\tilde{x}}}^{m+1}}+\eps_\delta^{m+1} h_\delta\at{\tilde{x}}\qquad \mbox{with}\quad \abs{h_\delta\at{\tilde{x}}}\leq C\,.
\end{equation}
Standard ODE arguments now imply
\begin{equation}
\label{Thm:TS.Asymptotics.PEqn1}
\sup_{0\leq\tilde{x}\leq\tilde{x}_*}\at{\abs{\tilde{S}_\delta\at{\tilde{x}}-
\tilde{S}_0\at{\tilde{x}}}+\abs{\tilde{S}_\delta^\prime\at{\tilde{x}}-
\tilde{S}_0^\prime\at{\tilde{x}}}}\leq C\eps_\delta^{m+1}
\end{equation}
for any fixed $\tilde{x}_*>0$, where $C$ depends on $\tilde{x}_*$.
\par
\ul{\emph{Properties of $\tilde{S}_\delta$}}: %
The monotonicity of both $\tilde{S}_\delta$ and $\tilde{S}_\delta^\prime$ on $I_\delta$ follows directly from the unimodality of $R_\delta$, $V_\delta$ and the definition \eref{Eqn:TipS.Def1}. In particular, $\tilde{S}_\delta$ is convex on $I_\delta$. We therefore  have
\begin{equation*}
\tilde{S}_\delta\at{\tilde{x}}\geq \tilde{S}_\delta\at{\tilde{x}_*}+\tilde{S}_\delta\at{\tilde{x}_*}^\prime\,\at{\tilde{x}-\tilde{x}_*}\,,
\end{equation*}
and choosing $\tilde{x}_*>0$ sufficiently close to $0$ we find a constant $c$ such that
\begin{equation}
\label{Thm:TS.Asymptotics.PEqn2}
1+\tilde{S}_\delta\at{x}\geq c\at{1+\tilde{x}}
\end{equation}
holds for all $\tilde{x}\in I_\delta$, where we used that $\lim_{\delta\to0}\tilde{S}_\delta^\prime\at{\tilde{x}_*}=\tilde{S}_0^\prime\at{\tilde{x}_*}>0$ is implied by \eref{Thm:TS.Asymptotics.PEqn1}. Notice that
\eref{Thm:TS.Asymptotics.PEqn2} holds also for $\delta=0$.
\par
\ul{\emph{Estimates for $\tilde{S}_\delta$}}: By \eref{Eqn:LimitIVP} and \eref{Thm:TS.Asymptotics.PEqn0} -- and using both
the Mean Value Theorem as well as \eref{Thm:TS.Asymptotics.PEqn2} -- we obtain
\begin{equation}
\label{Thm:TS.Asymptotics.PEqn3}
\abs{\tilde{S}_\delta^{\prime\prime}\at{\tilde{x}}
-\tilde{S}_0^{\prime\prime}\at{\tilde{x}}}\leq \frac{C}{\at{1+\tilde{x}}^{m+2}}\abs{\tilde{S}_\delta\at{\tilde{x}}
-\tilde{S}_0\at{\tilde{x}}}+C\eps_\delta^{m+1}\,.
\end{equation}
Integration with respect to $\tilde{x}$ yields
\begin{equation*}
\abs{\tilde{S}_\delta^{\prime}\at{\tilde{x}}
-\tilde{S}_0^{\prime}\at{\tilde{x}}}\leq
C\eps_\delta^{m+1}\tilde{x}+
\int_{0}^{\tilde{x}}\frac{C}{\at{1+\tilde{y}}^{m+2}}\int_0^{\tilde{y}}\abs{\tilde{S}_\delta^\prime\at{\tilde{z}}
-\tilde{S}_0^\prime\at{\tilde{z}}}\dint{\tilde{z}}\dint{\tilde{y}}
\end{equation*}
since \eref{Eqn:TipS.Id0} ensures that
\begin{equation*}
\tilde{S}_\delta\at{0}=\tilde{S}_0\at{0}=0\,,\qquad
\tilde{S}_\delta^\prime\at{0}=\tilde{S}_0^\prime\at{0}=0\,,
\end{equation*}
and a direct computation reveals
\begin{equation*}
\abs{\tilde{S}_\delta^{\prime}\at{\tilde{x}}
-\tilde{S}_0^{\prime}\at{\tilde{x}}}\leq
C\eps_\delta^{m+1}\tilde{x}+
\int_{0}^{\tilde{x}}\frac{C}{\at{1+\tilde{z}}^{m+1}}\abs{\tilde{S}_\delta^\prime\at{\tilde{z}}
-\tilde{S}_0^\prime\at{\tilde{z}}}\dint{\tilde{z}}\,.
\end{equation*}
Employing the Gronwall Lemma for $\tilde{x}\geq 0$ we obtain
\begin{equation}\label{eq:GRON}
\abs{\tilde{S}_\delta^{\prime}\at{\tilde{x}}
-\tilde{S}_0^{\prime}\at{\tilde{x}}}\leq
C\eps_\delta^{m+1} \tilde{x} \exp\at{\int_{0}^{\tilde{x}}\frac{C}{\at{1+\tilde{z}}^{m+1}}\dint{\tilde{z}}}\leq
C\eps_\delta^{m+1} \tilde{x} \,,
\end{equation}
and using $\tilde{x}\leq 1/\at{2\tilde\mu_\delta}$ as well as
the lower bound for $\mu_\delta$ from Lemma \ref{Lem:SimpleConv} we arrive at \eref{Thm:TS.Asymptotics.Eqn2}$_1$. Moreover,
integrating \eref{eq:GRON} with respect to $\tilde{x}$ gives
\begin{equation*}
\abs{\tilde{S}_\delta\at{\tilde{x}}
-\tilde{S}_0\at{\tilde{x}}}\leq  C \eps_\delta^{m+1} \tilde{x}^2
\end{equation*}
and hence \eref{Thm:TS.Asymptotics.Eqn1}. In combination with
\eref{Thm:TS.Asymptotics.PEqn3} we further get
\begin{equation*}
\abs{\tilde{S}_\delta^{\prime\prime}\at{\tilde{x}}
-\tilde{S}_0^{\prime\prime}\at{\tilde{x}}}\leq C\eps_\delta^{m+1}\at{
 \frac{\tilde{x}^2}{\at{1+\tilde{x}}^{m+2}}+1}\,,
\end{equation*}
which in turn provides \eref{Thm:TS.Asymptotics.Eqn2}$_2$.
\end{proof}
A first consequence of Theorem \ref{Thm:TS.Asymptotics} are leading order expressions for $\mu_\delta$ and hence for $\si_\delta$; below we improve this result by specifying the next-to-leading order corrections
in Corollary \ref{Cor:ScalingLaws}.
\begin{corollary}[convergence of $\mu_\delta$ and $\si_\delta$]
\label{Cor:TS.ConvPrms}
\quad
\begin{equation}
\label{Cor:TS.ConvPrms.Eqn1}
 \frac{\mu_\delta}{\eps_\delta} \quad\xrightarrow{\delta\to0}\quad \ol{\mu},\qquad\qquad
\si_\delta \eps_\delta^{m}\quad\xrightarrow{\delta\to0}\quad
 \ol\mu^2\,.
\end{equation}
\end{corollary}
\begin{proof}
By construction -- see \eref{Eqn:TipS.Def1} -- and  Lemma \ref{Lem:SimpleConv} we have
\begin{equation*}
\eps_\delta\,\tilde{S}_\delta\at{\frac{1}{2\mu_\delta}}=
1-\eps_\delta - R_\delta\at{\mbox{$\frac12$}}\quad
\xrightarrow{\;\;\delta\to{0}\;\;}\quad\mbox{$\frac12$}\,.
\end{equation*}
and in view of Theorem \ref{Thm:TS.Asymptotics} we conclude that
\begin{equation*}
2\, \eps_\delta\,\tilde{S}_0\at{\frac{1}{2\mu_\delta}}\quad
\xrightarrow{\;\;\delta\to{0}\;\;}\quad 1\,.
\end{equation*}
On the other hand, the estimates \eref{Lem:TS.LimitProfile.Eqn2a} and \eref{Lem:TS.LimitProfile.Eqn2b} evaluated at $\tilde{x}=1/\at{2\mu_\delta}$ imply
\begin{equation*}
2\,\mu_\delta\, \tilde{S}_0\at{\frac{1}{2\mu_\delta}}\quad
\xrightarrow{\;\;\delta\to{0}\;\;} \quad \ol\mu\,,
\end{equation*}
so the claim for $\mu_\delta$ follows immediately. The convergence result for $\si_\delta$ is a consequence of \eref{Eqn:DefEpsMu}.
\end{proof}
%
%
% ------------------------------------------------------------------
\section{Further asymptotic formulas}\label{sect:further}
% ------------------------------------------------------------------
%
In this section we exploit the asymptotic results on the tip scaling and
derive approximation formulas for the foot and transition scaling,
the tails of the profile functions, and for the scaling relations between $\delta$, $\eps_\delta$, and $\si_\delta$. In this way we obtain a complete set of asymptotic formulas which finally allows us to extract all relevant information
on the limit $\delta\to0$ from the function $\tilde{S}_0$ only. We also sketch a possible application of these formulas, namely the study of the linearized eigenvalue problem.
%
% ------------------------------------------------------------------
\subsection{Transition scaling \texorpdfstring{of $V_\delta$}{}}
% ------------------------------------------------------------------
%
%
To describe the jump-like behavior of $V_\delta$ near $x=\pm\frac12$ we
introduce the rescaled profiles $\tilde{W}_\delta$ by
\begin{equation}
\label{Eqn:TrS.Def1}
\tilde{W}_\delta\at{\tilde{x}}:=\frac{\mu_\delta}{\eps_\delta}
V_\delta\at{-\mbox{$\frac12$}+\mu_\delta\tilde{x}}\,
\end{equation}
and refer to Figure \ref{Fig.Convergence.TnFScaling} for an illustration. The key observation for the asymptotics of $\tilde{W}_\delta$ is
the approximation
\begin{equation*}
\tilde{W}_\delta^\prime\at{\tilde{x}}\approx \tilde{F}_\delta\at{\tilde{x}}\approx
\mbox{$\frac12$} \tilde{S}_\delta^{\prime\prime}\at{\tilde{x}}\,,
\end{equation*}
and hence we are able to prove the convergence of $\tilde{W}_\delta$ to
\begin{equation}
\label{Eqn:TrS.Lim}
\tilde{W}_0\at{\tilde{x}}:=\mbox{$\frac12$}\at{\tilde{S}_0^\prime\at{\tilde{x}}+\ol{\mu}}\,,
\end{equation}
where $\tilde{S}_0$ is the ODE solution from the tip scaling and defined in \eref{Eqn:LimitIVP}.
\begin{figure}[t!]
\centering{%
\includegraphics[width=0.95\textwidth]{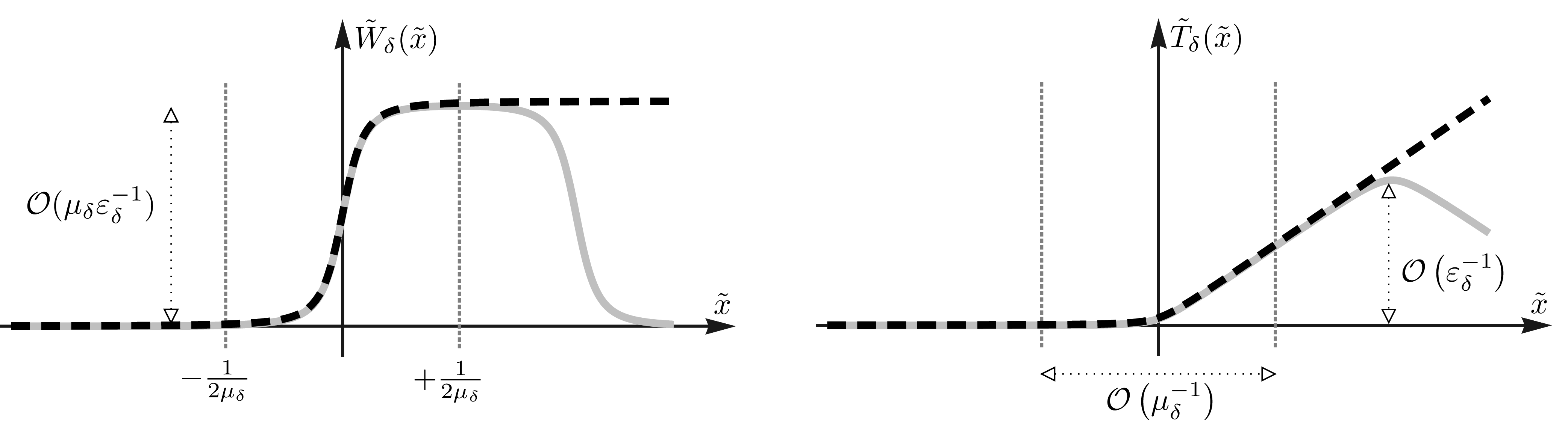}
}%
\caption{Cartoon of the tip and the foot scaling: The functions $\tilde{W}_\delta$ and $\tilde{T}_\delta$ for $\delta>0$ (gray, solid) and $\delta=0$ (black, dashed).
The limits $\tilde{W}_0$ and $\tilde{T}_0$ can be computed from $\tilde{S}_0$, see \eref{Eqn:TrS.Lim} and \eref{Eqn:FoS.Lim}.} %
\label{Fig.Convergence.TnFScaling}%
\end{figure}
\begin{theorem}[convergence under the transition scaling]
\label{Thm:TrS.MR}
We have
\begin{equation}
\label{Thm:TrS.MR.Eqn1}
\sup_{x\in{J_\delta}}
\babs{\tilde{W}_\delta\at{\tilde{x}}-\tilde{W}_0\at{\tilde{x}}}\leq C\eps_\delta^m
\,,\qquad %
\sup_{x\in{J_\delta}}\babs{\tilde{W}_\delta^\prime\at{\tilde{x}}-\tilde{W}_0^\prime\at{\tilde{x}}}\leq C\eps_\delta^{m+1}
\end{equation}
for some constant $C$ independent of $\delta$.
\end{theorem}
\begin{proof}
\ul{\emph{Uniform estimates for the derivative}}: %
By \eref{Eqn:TrS.Def1}, the traveling wave equation \eref{Eqn:TW.Diff}, and
the definition of $\mu_\delta$ in \eref{Eqn:DefEpsMu} we have
\begin{eqnarray*}
\tilde{W}_\delta^\prime\at{\tilde{x}}&=\frac{\mu^2_\delta}{\eps_\delta} V_\delta^\prime\at{-\mbox{$\frac12$} +\mu_\delta \tilde{x}}
\\&=
\si_\delta \eps_\delta^{m+1}V_\delta^\prime\at{-\mbox{$\frac12$} +\mu_\delta \tilde{x}}
\\&=
\eps_\delta^{m+1}\Bat{\Phi^\prime\bat{R_\delta\at{\mu_\delta\tilde{x}}}-
\Phi^\prime\bat{R_\delta\at{-1+\mu_\delta\tilde{x}}}}\,,
\end{eqnarray*}
and from \eref{Eqn:TipS.Def2},  \eref{Eqn:TipS.Def3} as well as \eref{Eqn:TipS.Id1} we infer that
\begin{equation*}
\tilde{W}_\delta^\prime\at{\tilde{x}}=\tilde{F}_\delta\at{\tilde{x}}-\tilde{G}_\delta\at{\tilde{x}}
=\mbox{$\frac12$}\Bat{\tilde{S}_\delta^{\prime\prime}\at{\tilde{x}}-
\tilde{G}_\delta\at{\tilde{x}}+\tilde{G}_\delta\at{-\tilde{x}}}\,.
\end{equation*}
The estimate \eref{Thm:TrS.MR.Eqn1}$_2$ is now a direct consequence of
\eref{Eqn:TS.GEstimates} and Theorem \ref{Thm:TS.Asymptotics}.
\par
\ul{\emph{Pointwise estimate at $\tilde{x}=1/\at{2\mu_\delta}$}}: %
Using \eref{Eqn:TrS.Def1}, the traveling wave equation \eref{Eqn:TW.Int}, and the
unimodality of $R_\delta$ we obtain
\begin{eqnarray*}
\frac{\eps_\delta}{\mu_\delta}
\tilde{W}_\delta\at{\frac{1}{2\mu_\delta}}= V_\delta\at{0}&=&\int_{-\mbox{$\frac12$}}^{+\mbox{$\frac12$}}\frac{\Phi^\prime\bat{R_\delta\at{x}}}{\si_\delta}\dint{x}
\\&=&\frac{2}{\si_\delta\,\eps_\delta^{m+1}}\int_{0}^{+\mbox{$\frac12$}}\eps_\delta^{m+1}\,\Phi^\prime\bat{R_\delta\at{x}}\dint{x}\,.
\end{eqnarray*}
Thanks to \eref{Eqn:DefEpsMu}, \eref{Eqn:TipS.Def2}, and \eref{Eqn:TipS.Id1} we thus conclude
\begin{eqnarray*}
\tilde{W}_\delta\at{\frac{1}{2\mu_\delta}}=\frac{\mu_\delta^2}{\si_\delta \,\eps_\delta^{m+2}}\int_{I_\delta}2\,\tilde{F}_\delta\at{\tilde{x}}\dint{\tilde{x}}
&=&
\int_{I_\delta}2\,\tilde{F}_\delta\at{\tilde{x}}\dint{\tilde{x}}
\\&=& \int_{I_\delta}\tilde{S}_\delta^{\prime\prime}\at{\tilde{x}}-\tilde{G}_\delta\at{\tilde{x}}-\tilde{G}_\delta\at{-\tilde{x}}\dint{\tilde{x}}
\end{eqnarray*}
and since $\int_{I_\delta} \tilde{G}_\delta\at{\pm\tilde{x}}\dint{\tilde{x}}{\,=\,}\DO{\eps_\delta^m}$ holds according to \eref{Eqn:TS.GEstimates} and Corollary \ref{Cor:TS.ConvPrms}, we obtain
\begin{equation*}
\tilde{W}_\delta\at{\frac{1}{2\mu_\delta}}
=
S_\delta^\prime\at{\frac{1}{2\mu_\delta}}+\DO{\eps_\delta^m}\,,
\end{equation*}
Lemma \ref{Lem:TS.LimitProfile} combined with Theorem \ref{Thm:TS.Asymptotics} and \eref{Cor:TS.ConvPrms.Eqn1} finally provides
\begin{equation*}
\tilde{W}_\delta\at{\frac{1}{2\mu_\delta}}=\ol{\mu}+\DO{\eps_\delta^m}=
\tilde{W}_0\at{\frac{1}{2\mu_\delta}}
+\DO{\eps_\delta^m}\,,
\end{equation*}
so \eref{Thm:TrS.MR.Eqn1}$_1$ follows from \eref{Thm:TrS.MR.Eqn1}$_2$.
\end{proof}
%
%
%
% ------------------------------------------------------------------
\subsection{Tail estimates}
% ------------------------------------------------------------------
%
%
%
We complement our previous results by estimates for the tails of
the profiles $V_\delta$ and $R_\delta$. The derivation of those exploits
the exponential decay with respect to $x$, which we establish
by adapting an idea from \cite{HR10}.
\begin{theorem}[tail estimates for $V_\delta$ and $R_\delta$]
\label{Thm:TailEstimates}
The estimates
\begin{equation*}
\sup\limits_{\abs{x}\geq1} V_\delta\at{x}\;\;+\;\;\int_{\abs{x}\geq1}V_\delta\at{x}\dint{x}\;\leq\; C\eps_\delta^{m}
\end{equation*}
and
\begin{equation*}
\sup\limits_{\abs{x}\geq\mbox{$\frac32$}} R_\delta\at{x}\;\;+\;\;
\int_{\abs{x}\geq\mbox{$\frac32$}}R_\delta\at{x}\dint{x}\;\leq \; C\eps_\delta^{m}
\end{equation*}
hold for some constant $C$ independent of $\delta$.
\end{theorem}
\begin{proof}
Since $V_\delta$ is unimodal and nonnegative by Assumption \ref{Ass:Waves}, we have
\begin{equation}
\label{Thm:TailEstimates.PEqn1}
0\leq V_\delta\at{x}\leq V_\delta\at{-1}\quad \mbox{for all}\quad x\leq-1.
\end{equation}
Moreover, the properties of the convolution operator $A$ combined with
\eref{Eqn:TW.Int} imply
\begin{equation}
\label{Thm:TailEstimates.PEqn5}
R_\delta\at{x}\leq V_\delta\at{x+\mbox{$\frac12$}}\,,\quad V_\delta\at{x}\leq \frac{\Phi^\prime\Bat{R_\delta\at{x+\mbox{$\frac12$}}}}{\si_\delta}
\quad \mbox{for all}\quad x<-1
\end{equation}
and we infer that
\begin{equation}
\label{Thm:TailEstimates.PEqn2}
V_\delta\at{x}\leq \frac{\Phi^\prime\bat{V_\delta\at{x+1}}}{\si_\delta}\leq
C{\eps_\delta}^m V_\delta\at{x+1}\quad \mbox{for all}\quad x<-2
\end{equation}
where we used Corollary \ref{Cor:TS.ConvPrms}, the properties of $\Phi^\prime$, and that the unimodality of $V_\delta$ guarantees
\begin{equation*}
2 V_\delta\at{-1}^2\leq \norm{V_\delta}_2^2=\at{1-\delta}^2\qquad\mbox{and hence}\qquad
V_\delta\at{-1}\leq 1/\sqrt{2}<1\,.
\end{equation*}
By iteration of \eref{Thm:TailEstimates.PEqn2} we obtain
\begin{equation*}
V_\delta\at{x-n}\leq\bat{C\eps_\delta^m}^n V_\delta\at{x}
\qquad \mbox{for all}\quad x<-1\,,\; n\in\Nset
\end{equation*}
and conclude that $V_\delta$ decays exponentially with rate
\begin{equation*}
\la_\delta\geq
m\abs{\ln \eps_\delta}\at{1+\Do{1}}
\end{equation*}
and satisfies
\begin{equation}
\label{Thm:TailEstimates.PEqn3}
\int\limits_{-\infty}^{-1}{V_\delta}\at{x}\dint{x}\leq C V_\delta\at{-1}\,.
\end{equation}
Finally, \eref{Eqn:TrS.Def1} and Theorem \ref{Thm:TrS.MR} ensure that
\begin{equation}
\label{Thm:TailEstimates.PEqn4}
\eqalign{ %
 V_\delta\at{-1}
&= %
\frac{\eps_\delta}{\mu_\delta} \tilde{W}_\delta\at{-\frac{1}{2\mu_\delta}}
\\ %
&= \frac{\eps_\delta}{\mu_\delta} \at{\frac{ \ol{\mu}+\tilde{S}_0^\prime\at{\displaystyle-\frac{1}{2\mu_\delta}}}{2}+\DO{\eps_\delta^m}}=
\DO{\eps_\delta^m}%
} %
\end{equation}
where the last identity stems from \eref{Lem:TS.LimitProfile.Eqn2a} and \eref{Cor:TS.ConvPrms.Eqn1}. The desired estimates for $V_\delta$ are now direct consequences of \eref{Thm:TailEstimates.PEqn1}, \eref{Thm:TailEstimates.PEqn3}, and \eref{Thm:TailEstimates.PEqn4}, and imply the claim for $R_\delta$ thanks to \eref{Thm:TailEstimates.PEqn5}$_1$.
\end{proof}
Notice that the exponential decay rate in the proof of Theorem \ref{Thm:TailEstimates} is asymptotically optimal for small $\delta$. In fact, after linearization of $\Phi^\prime$ in $0$ we find the tail identity
\begin{equation*}
\si_\delta V_\delta \approx A^2 V_\delta\,,
\end{equation*}
and the usual exponential ansatz predicts that the exact decay rate $\la_\delta$ is the positive solution to the transcendental equation
\begin{equation*}
\frac{\mathrm{sinh}\,\at{\la_\delta/2}}{\la_\delta/2}=\sqrt{\si_\delta}= \ol{\mu}\eps_\delta^{-m/2}\Bat{1+\DO{\eps_\delta}}\,.
\end{equation*}
In particular, $\la_\delta$ is large for small $\delta$ and satisfies $\la_\delta=m\abs{\ln\eps_\delta}+\DO{\ln\abs{\ln\eps_\delta}}$.
%
%
%
% ------------------------------------------------------------------
\subsection{Relations between the parameters\texorpdfstring{ $\delta$, $\eps_\delta$, and $\mu_\delta$}{}}
% ------------------------------------------------------------------
%
In this section we identify the scaling relations between the small quantities
\begin{equation*}
\delta,\, \quad \eps_\delta\,,\quad \mu_\delta
\end{equation*}
and start with two auxiliary results.
\begin{lemma}[scaling relation between $\mu_\delta$ and $\eps_\delta$]
\label{Lem:ScalingMu}
The formula
\begin{equation*}
{\mu_\delta}=\frac{\ol\mu \, \eps_\delta}{1+\eps_\delta\at{\ol\ka-1}}+\DO{\eps_\delta^{m+1}}
\end{equation*}
holds for all $0<\delta<1$.
\end{lemma}
\begin{proof}
Our starting point is the identity
\begin{eqnarray*}
1-\eps_\delta &=& R_\delta\at{0}=2\int_{-1/2}^0 V_\delta\at{x}\dint{x}\\&=&2\,{\mu_\delta}\int_{0}^{1/\at{2\mu_\delta}}
V_\delta\at{-\frac12+\tilde{\mu}\tilde{x}}\dint{\tilde{x}}=2\,{\eps_\delta}\int_{0}^{1/\at{2\mu_\delta}}
\tilde{W}_\delta\at{\tilde{x}}\dint{\tilde{x}}
\end{eqnarray*}
which follows from \eref{Eqn:TW.Int}, Assumption \ref{Ass:Waves}, and \eref{Eqn:TrS.Def1}. Theorem \ref{Thm:TrS.MR} now yields
\begin{eqnarray*}
1-\eps_\delta &=&
2\,{\eps_\delta}\int_{0}^{1/\at{2\mu_\delta}}
\tilde{W}_0\at{\tilde{x}}\dint{x}+\DO{\eps_\delta^{m}}
\\&=&
\frac{\eps_\delta}{\mu_\delta}\at{\frac{\ol\mu}{2}+\mu_\delta \tilde{S}_0\at{\frac{1}{2\mu_\delta}}}+\DO{\eps_\delta^{m}}\,,
\end{eqnarray*}
while \eref{Lem:TS.LimitProfile.Eqn2b} and \eref{Cor:TS.ConvPrms.Eqn1} provide
\begin{equation*}
\tilde{S}_0\at{\frac{1}{2\mu_\delta}}=  \frac{1}{2\mu_\delta}\tilde{S}_0^\prime\at{\frac{1}{2\mu_\delta}}-\ol\ka+\DO{\mu_\delta^{m-1}}=\frac{\ol{\mu}}{2\mu_\delta}-\ol\ka+\DO{\eps_\delta^{m-1}}\,.
\end{equation*}
The combination of the latter two formulas yields
\begin{equation*}
1-\eps_\delta=\frac{\eps_\delta}{\mu_\delta}\Bat{\ol\mu-\mu_\delta\ol{\ka}}+\DO{\eps_\delta^{m}}\,,
\end{equation*}
and rearranging terms we find the desired result.
\end{proof}
\begin{lemma}[scaling relation between $\eps_\delta$ and $\delta$]
\label{Lem:ScalingDe}
We have
\begin{equation*}
\delta = 1- \sqrt{\Bat{1+\eps_\delta\at{\ol{\ka}-1}}
\at{1-\eps_\delta\at{1+\frac{\ol{\eta}}{\ol\mu}}}}+\DO{\eps_\delta^m}\,,
\end{equation*}
where the right hand side is real-valued and positive for all sufficiently small $\delta>0$.
\end{lemma}
\begin{proof}
From Assumption \ref{Ass:Waves} and the tail estimates in Theorem \ref{Thm:TailEstimates} we derive
\begin{equation*}
\at{1-\delta}^2=\int_\Rset V_\delta\at{x}^2\dint{x}=2\int_{-1}^0V_\delta\at{x}^2\dint{x}+\DO{\eps_\delta^m}\,,
\end{equation*}
and \eref{Eqn:TrS.Def1} gives
\begin{equation*}
\int_{-1}^0V_\delta\at{x}^2\dint{x}=\mu_\delta \int\limits_{J_\delta}V_\delta\at{-\frac12+\mu_\delta\tilde{x}}^2\dint{\tilde{x}}
=\frac{\eps_\delta^2}{\mu_\delta}\int\limits_{J_\delta}\tilde{W}_\delta\at{\tilde{x}}^2\dint{\tilde{x}}\,.
\end{equation*}
 Therefore -- and thanks to Theorem~\ref{Thm:TrS.MR} -- we get
\begin{equation}
\label{Lem:ScalingDe.PEqn1}
\eqalign{ %
\at{1-\delta}^2
&= %
2\,\frac{\eps_\delta^2}{\mu_\delta}\int\limits_{J_\delta}\tilde{W}_0\at{\tilde{x}}^2\dint{\tilde{x}}+\DO{\eps_\delta^{m}}
\cr&=
\frac{\ol{\mu}^2\eps_\delta^2}{2\mu_\delta^2}+\frac{\eps_\delta^2}{\mu_\delta}
\int_{I_\delta}\tilde{S}_0^\prime\at{\tilde{x}}^2 \dint{\tilde{x}}+\DO{\eps_\delta^{m}}
} %
\end{equation}
where we used that $\tilde{S}_0^\prime$ is an odd function, and
integration by parts yields
\begin{eqnarray*}
\int_{I_\delta}\bat{\tilde{S}_0^\prime\at{\tilde{x}}}^2 \dint{\tilde{x}}&=&
\tilde{S}_0\at{\frac{1}{2\mu_\delta}}
\tilde{S}_0^\prime\at{\frac{1}{2\mu_\delta}}
-\int_{I_\delta}\tilde{S}_0\at{\tilde{x}}\, \tilde{S}_0^{\prime\prime}\at{\tilde{x}} \dint{\tilde{x}}
\\&=&
\at{\frac{1}{2\mu_\delta}\tilde{S}_0^\prime\at{\frac{1}{2\mu_\delta}}-\ol{\ka}}\,\tilde{S}_0^\prime\at{\frac{1}{2\mu_\delta}}
-
\ol{\eta}+\DO{\mu_\delta^{m-1}}
\\&=&
\at{\frac{\ol\mu}{2\mu_\delta}-\ol{\ka}}\ol{\mu}-\ol{\eta}+\DO{\mu_\delta^{m-1}}
\end{eqnarray*}
thanks to the estimates and decay results from Lemma \ref{Lem:TS.LimitProfile}. In summary we find
\begin{equation*}
\at{1-\delta}^2=\frac{\ol{\mu}^2\eps_\delta^2}{\mu_\delta^2}-
\frac{\eps_\delta^2}{\mu_\delta}\at{\ol{\ka}\,\ol{\mu}+\ol{\eta}}+
\DO{\eps_\delta^m}
\end{equation*}
and eliminating $\mu_\delta$ by Lemma \ref{Lem:ScalingMu} yields via
\begin{equation}
\label{Lem:ScalingDe.PEqn2}
\at{1-\delta}^2= \Bat{1+\eps_\delta\at{\ol{\ka}-1}}^2 - \eps_\delta \Bat{1+\eps_\delta\at{\ol{\ka}-1}}\at{\ol{\ka}+\frac{\ol{\eta}}{\ol{\mu}}}
+\DO{\eps_\delta^m}
\end{equation}
the assertion.
\end{proof}
 Lemma \ref{Lem:ScalingMu} and Lemma \ref{Lem:ScalingDe} provide explicit formulas for
\begin{equation*}
\mu_\delta\sim \eps_\delta  \qquad \mbox{and}\qquad
\delta\sim\eps_\delta
\end{equation*}
in terms of $\eps_\delta$ but there is a
slight mismatch between both results since the error bounds in
the formula for $\mu_\delta$ are of higher order than those in the scaling law for $\delta$. It is not clear, at least to the authors, whether this mismatch concerns
the real error terms or just means that the bounds in Lemma \ref{Lem:ScalingDe} are less optimal than those in
Lemma \ref{Lem:ScalingMu}.
\par
Our main result concerning the scaling relations between the different
parameters can now be formulated as follows.
\begin{corollary}[leading order scaling laws]
\label{Cor:ScalingLaws}
The relation between $\mu_\delta$ and $\eps_\delta$ can be computed
up to error terms of order $\DO{\eps_\delta^{m+1}}$, while the scaling law between
$\delta$ and $\eps_\delta$ is determined up to order $\DO{\eps_\delta^{m}}$ only. In particular, we have
\begin{eqnarray*}
\mu_\delta &=& \ol{\mu}\,\eps_\delta+\ol\mu\,\at{1-\ol\ka}\,\eps_\delta^2+\Do{\eps_\delta^2}\,,\\
\si_\delta &=&
\ol{\mu}^2 \eps_\delta^{-m}+2\,\ol{\mu}^2\, \at{1-\ol\ka}\,\eps_\delta^{-m+1}
+\Do{\eps_\delta^{-m+1}}
\end{eqnarray*}
for all $m>1$, as well as
\begin{equation*}
\delta = \frac{2\,\ol{\mu}-\ol{\mu}\,\ol{\ka}-\ol{\eta}}{2\,\ol{\mu}}\eps_\delta\,+
\frac{\ol{\mu}^2\,\ol{\ka}^2+\ol{\eta}^2+2\,\ol{\mu}\,\ol{\ka}\,\ol{\eta}}{8\,\ol{\mu}^2}\,\eps_\delta^2+\Do{\eps_\delta^2}
\end{equation*}
provided that $m>2$.
\end{corollary}

\begin{proof}
All assertions are provided by Lemma \ref{Lem:ScalingMu}, Lemma \ref{Lem:ScalingDe}, and formula \eref{Eqn:DefEpsMu}.
\end{proof}
%
%

% ------------------------------------------------------------------
\subsection{Foot scaling \texorpdfstring{of $R_\delta$}{}}
% ------------------------------------------------------------------
%
%
We study now the asymptotic behavior of $R_\delta$ near $x=\pm1$.  To this end
we define
\begin{equation}
\label{Eqn:FoS.Def1}%
\tilde{T}_\delta\at{\tilde{x}}:=\frac{R_\delta\at{-1+\mu_\delta\tilde{x}}}{\eps_\delta}
\end{equation}
and find by direct calculations the identity
\begin{equation}
\label{Eqn:FoS.Id1}%
\tilde{T}_\delta^{\prime\prime}\at{\tilde{x}}=\tilde{F}_\delta\at{\tilde{x}}+
\tilde{H}_\delta\at{\tilde{x}}-2\tilde{G}_\delta\at{\tilde{x}}\approx
\mbox{$\frac12$} \tilde{S}_\delta^{\prime\prime}\at{\tilde{x}}
\end{equation}
because both  $\tilde{G}_\delta$ and
\begin{equation*}
\tilde{H}_\delta\at{\tilde{x}}:=\Phi^\prime\bat{R_\delta\at{-2+\mu_\delta\tilde{x}}}
\end{equation*}
can be neglected on the interval $I_\delta$. We further define
\begin{equation}
\label{Eqn:FoS.Lim}
\tilde{T}_0\at{\tilde{x}}:=\mbox{$\frac12$}\at{\tilde{S}_0\at{\tilde{x}}+\ol{\mu}\,\tilde{x}+\ol{\ka}}
\end{equation}
and show that $\tilde{T}_\delta$ converges as $\delta\to0$ to $\tilde{T}_0$, see Figure \ref{Fig.Convergence.TnFScaling}.

\begin{lemma}[asymptotics of $R_\delta$ at $x=\pm\frac12$ and for $V_\delta$ at $x=0$]
\label{Lem:PointwiseAsymp}
The terms $R_\delta^\prime\at{-\frac12}$, $2R_\delta\at{\frac12}$, and
$ V_\delta\at{0}$ are identical to leading order in $\delta$.
 More precisely, we have
\begin{equation*}
 R_\delta^\prime\at{\mbox{$\frac12$}}=
 \frac{\eps_\delta\, \ol\mu}{\mu_\delta }+\DO{\eps_\delta^m}=\bat{1+\eps_\delta\at{\ol{\ka}-1}}+\DO{\eps_\delta^m}\,,
\end{equation*}
and
\begin{equation*}
\babs{2\,R_\delta\at{\mbox{$\frac12$}}-R_\delta^\prime\at{-\mbox{$\frac12$}}}=
\DO{\eps_\delta^m}\,,\qquad\babs{V_\delta\at{0}-R_\delta^\prime\at{\mbox{-$\frac12$}}}
=\DO{\eps_\delta^m}
\end{equation*}
for all $0<\delta<1$.
\end{lemma}
\begin{proof}
From \eref{Eqn:TipS.Def1},  Lemma \ref{Lem:TS.LimitProfile}, Theorem \ref{Thm:TS.Asymptotics}, and \eref{Cor:TS.ConvPrms.Eqn1} we infer
\begin{eqnarray*}
R_\delta^\prime\at{-\mbox{$\frac12$}}&=&-R_\delta^\prime\at{\mbox{$\frac12$}}=\frac{\eps_\delta}{\mu_\delta}\tilde{S}_\delta^\prime\at{\frac{1}{2\mu_\delta}}\\&=&\frac{\eps_\delta}{\mu_\delta}\tilde{S}_0^\prime\at{\frac{1}{2\mu_\delta}}+\DO{\mu_\delta^m}=
\frac{\eps_\delta\,\ol{\mu}}{\mu_\delta}+\DO{\eps_\delta^m}
\end{eqnarray*}
and similarly
\begin{eqnarray*}
R_\delta\at{\mbox{$\frac12$}}&=
1-\eps_\delta-\eps_\delta\tilde{S}_\delta\at{\frac{1}{2\mu_\delta}} =1-\eps_\delta-\eps_\delta\tilde{S}_0\at{\frac{1}{2\mu_\delta}}+\DO{\mu_\delta^m}
\\&=1-\eps_\delta-\eps_\delta\at{\frac{1}{2\mu_\delta}\tilde{S}_0^\prime\at{\frac{1}{2\mu_\delta}}-\ol{\ka}}+\DO{\eps_\delta^m}
\\&=1+\eps_\delta\at{\ol\ka-1}-\frac{\eps_\delta\,\ol{\mu}}{2\,\mu_\delta}+\DO{\eps_\delta^m}\,.
\end{eqnarray*}
Thanks to \eref{Eqn:TrS.Def1} and Theorem \ref{Thm:TrS.MR} we also find
\begin{equation*}
V_\delta\at{0}=\frac{\eps_\delta}{\mu_\delta}\tilde{W}_0\at{\frac{1}{2\mu_\delta}}+\DO{\eps_\delta^m}=\frac{\eps_\delta\,\ol{\mu}}{\mu_\delta}+\DO{\eps_\delta^m}\,,
\end{equation*}
and the result follows from Lemma \ref{Lem:ScalingMu}.
\end{proof}
\begin{theorem}[convergence under the foot scaling]
\label{Thm:FoS.MR}
The estimate
\begin{equation*}
\sup\limits_{\tilde{x}\in  J_\delta}\Bat{ %
\eps^2\,\babs{\tilde{T}_\delta\at{\tilde{x}}-\tilde{T}_0\at{\tilde{x}}}+
\eps\,\babs{\tilde{T}_\delta^\prime\at{\tilde{x}}-\tilde{T}_0^\prime\at{\tilde{x}}}+
\babs{\tilde{T}_\delta^{\prime\prime}\at{\tilde{x}}-\tilde{T}_0^{\prime\prime}\at{\tilde{x}}}
}\leq C\eps^{m+1}
\end{equation*}
holds with some constant $C$ independent of $\delta$. In particular, we have
\begin{equation*}
R_\delta\at{\pm1} = \mbox{$\frac12$}  \ol{\ka}\, \eps_\delta+\DO{\eps_\delta^m}\,.
\end{equation*}
\end{theorem}
\begin{proof}
The unimodality of $R_\delta$, the monotonicity of $\Phi^\prime$, and \eref{Eqn:TS.GEstimates} imply
\begin{equation*}
0\leq \tilde{H}_\delta\at{\tilde{x}}\leq \tilde{G}_\delta\at{\tilde{x}}\leq C\eps_\delta^{m+1}
\end{equation*}
for all $\tilde{x}\in J_\delta$. Combining this with \eref{Eqn:FoS.Id1} and Theorem \ref{Thm:TS.Asymptotics} we arrive at the desired estimates for the second derivatives. We also notice that  Lemma \ref{Lem:TS.LimitProfile} along with Lemma \ref{Lem:PointwiseAsymp} imply
\begin{eqnarray*}
\tilde{T}_0^\prime\at{\frac{1}{2\mu_\delta}}
&=&\mbox{$\frac12$}\tilde{S}_0^\prime\at{\frac{1}{2\mu_\delta}}+ \mbox{$\frac12$}\ol\mu
=\ol\mu+\DO{\eps_\delta^m}\\&=&\frac{\mu_\delta}{\eps_\delta}R_\delta^\prime\at{\mbox{$\frac12$}}+\DO{\eps_\delta^m}=
\tilde{T}_\delta^\prime\at{\frac{1}{2\mu_\delta}}+\DO{\eps_\delta^m}\,,
\end{eqnarray*}
where the last identity stems from \eref{Eqn:FoS.Def1}, and by similar arguments we justify
\begin{eqnarray*}
\tilde{T}_0\at{\frac{1}{2\mu_\delta}}&=&\frac12 \tilde{S}_0\at{\frac{1}{2\mu_\delta}}
+\frac14\frac{\ol{\mu}}{\mu_\delta}+\frac12\ol{\ka}
\\&=&\frac12\at{\frac{\ol{\mu}}{2\mu_\delta}-\ol{\ka}}+
\frac14\frac{\ol{\mu}}{\mu_\delta}+\frac12\ol{\ka}+\DO{\eps_\delta^{m-1}}
\\&=&\frac{\ol{\mu}}{2\mu_\delta}+\DO{\eps_\delta^{m-1}}\\&=&
\frac{R_\delta\at{\pm\frac12}}{\eps_\delta}+\DO{\eps_\delta^{m-1}}
=
\tilde{T}_\delta\at{\frac{1}{2\mu_\delta}}+\DO{\eps_\delta^{m-1}}\,.
\end{eqnarray*}
The assertions for the first and zeroth derivatives can thus be derived from the estimates for the second derivatives by
integration with respect to $\tilde{x}$.
\end{proof}
%
%
% ------------------------------------------------------------------
\subsection{Summary on the asymptotic analysis}
% ------------------------------------------------------------------
%
%
We finally combine all partial results as follows.
\begin{theorem}[global approximation in the high-energy limit]
\label{Thm:GlobApp}
The formulas
\begin{equation*}
\hat{R}_\eps\at{x}:=\left\{
\begin{array}{lcl}
1-\eps-\eps\,\tilde{S}_0\at{\displaystyle\frac{\abs{x}}{\hat{\mu}_\eps}}&&\mbox{for $0\leq\abs{x}<\frac12$}\\
\eps\,\tilde{T}_0\at{\displaystyle\frac{1-\abs{x}}{\hat{\mu}_\eps }}&&\mbox{for $\frac12\leq\abs{x}<\frac32$}\\
0&&\mbox{else}
\end{array}\right.\
\end{equation*}
and
\begin{equation*}
\hat{V}_\eps\at{x}:=\frac{\eps}{\hat{\mu}_\eps}\left\{
\begin{array}{lcl}
\tilde{W}_0\at{\displaystyle\frac{\frac12-\abs{x}}{\hat{\mu}_\eps }}&&\mbox{for $0\leq\abs{x}<1$}\\
0&&\mbox{else}
\end{array}\right.
\end{equation*}
with
\begin{equation*}
\hat{\mu}_\eps := \frac{\ol{\mu}\,\eps}{1+\eps\,\at{\ol\ka-1}}\,\qquad \hat{\si}_\eps := \eps^{-m-2}\,\hat{\mu}_\eps^2
\end{equation*}
approximate the solitary waves from Assumption \ref{Ass:Waves} in the sense of
\begin{equation*}
\bnorm{R_\delta-\hat{R}_{\eps_\delta}}_q +\bnorm{V_\delta-\hat{V}_{\eps_\delta}}_q  +\eps_\delta^m \babs{\si_{\delta}-\hat{\si}_{\eps_\delta}}=\DO{\eps_\delta^m}=\DO{\delta^m}
\end{equation*}
for any $q\in\ccinterval{1}{\infty}$. Here,
$\tilde{S}_0$ solves the ODE initial value problem \eref{Eqn:LimitIVP},
the constants $\ol{\mu}$, $\ol{\ka}$ are given in \eref{Lem:TS.LimitProfile.Consts}, and the functions $\tilde{W}_0$, $\tilde{T}_0$ are defined in \eref{Eqn:TrS.Lim}, \eref{Eqn:FoS.Lim}.
\end{theorem}
\begin{proof}
Notice that $\DO{\delta}=\DO{\eps_\delta}$ is ensured by Lemma \ref{Lem:ScalingDe} and that
it suffices to consider the case $q=\infty$ because the functions $\hat{R}_\eps$ and  $\hat{V}_\eps$ are compactly supported. Theorems \ref{Thm:TS.Asymptotics},
\ref{Thm:TrS.MR}, and \ref{Thm:FoS.MR} -- which concern the convergence under the different rescalings --- as well as the tail estimates from Theorem \ref{Thm:TailEstimates}
provide a variant of the desired estimates in which
$\hat{\mu}_{\eps_\delta}$  is replaced by
$\mu_\delta$. Thanks to Lemma \ref{Lem:ScalingMu} we also have
$\mu_\delta\sim\delta$
as well as
\begin{equation}
\label{Thm:GlobApp.PEqn1}
\mu_\delta = \hat{\mu}_{\eps_\delta}+\DO{\delta^{m+1}}\qquad\mbox{and hence}\qquad
\frac{1}{\mu_\delta}=\frac{1}{\hat{\mu}_{\eps_\delta}}+\DO{\delta^{m-1}}.
\end{equation}
Since $\tilde{S}_0^\prime$ is bounded, the intermediate value theorem implies
\begin{equation*}
\eps_\delta\abs{\tilde{S}_0\at{\frac{\abs{x}}{\mu_\delta}}-
\tilde{S}_0\at{\frac{\abs{x}}{\hat{\mu}_{\eps_\delta}}}}=
\eps_\delta\abs{\norm{\tilde{S}_0^\prime}_{\infty}\,\abs{x}\,\DO{\delta^{m-1}}}=\DO{\delta^m}\,,
\end{equation*}
 and by similar arguments we derive the corresponding estimate for $\tilde{T}_0$.
For the approximation of the velocity profile, the crucial estimate is
\begin{eqnarray*}
\abs{\tilde{W}_0\at{\frac{\mbox{$\frac12$}-\abs{x}}{\mu_\delta}}-
\tilde{W}_0\at{\frac{\mbox{$\frac12$}-\abs{x}}{\hat{\mu}_{\eps_\delta}}}}=
\abs{
\tilde{W}_0^\prime\at{\xi}\at{\mbox{$\frac12$}-\abs{x}}\DO{\delta^{m-1}}}=\DO{\delta^m},
\end{eqnarray*}
where $\xi$ denotes an intermediate value and where we used that the function $\tilde{x}\mapsto\tilde{x} \tilde{W}_0^\prime\at{\tilde{x}}$ is bounded. Finally, the estimates for $\si_\delta-\hat{\si}_{\eps_\delta}$ follow from \eref{Eqn:DefEpsMu} and \eref{Thm:GlobApp.PEqn1}.
\end{proof}
For practical purposes it might be more convenient to regard $\eps$ as the
independent parameter and $\delta$ as the derived quantity.
In this case we can employ the following result, which is, however,
weaker than Theorem \ref{Thm:GlobApp} since the guaranteed error bounds are of lower order.
\begin{corollary}[variant of the global approximation result]
\label{Cor:GlobApp}
We have
\begin{equation*}
\bnorm{\hat{R}_{\eps}-R_{\hat{\delta}_\eps}}_q +
\bnorm{\hat{V}_{\eps}-V_{\hat{\delta}_\eps}}_q  +\eps^{m-1} \babs{\hat{\si}_{\eps}-\si_{\hat{\delta}_\eps}}=\DO{\eps^{m-1}}
\end{equation*}
for any $q\in\ccinterval{1}{\infty}$, where $\hat{\delta}_\eps:=1-\norm{\hat{V}_\eps}_2\sim \eps$.
\end{corollary}
\begin{proof} Let $\eps_*$ be fixed,
where the subscript $*$ has been introduced for the sake of clarity only, and write $\delta_*:=\hat{\delta}_{\eps_*}$
as well as $\mu_*:=\hat{\mu}_{\eps_*}$ for
the quantities that can be computed directly and explicitly from $\eps_*$.
Lemma \ref{Lem:ScalingDe} provides
\begin{equation*}
\at{1-\delta_*}^2=\norm{V_{\delta_*}}_2^2=\Bat{1+\eps_{\delta_*}\at{\ol{\ka}-1}}
\at{1-\eps_{\delta_*}\at{1+\frac{\ol{\eta}}{\ol\mu}}}+\DO{\eps_{\delta_*}^m}\,,
\end{equation*}
where $\eps_{\delta_*}$ and $\mu_{\delta_*}$ are defined by
the exact wave data $\triple{R_{\delta_*}}{V_{\delta_*}}{\si_{\delta_*}}$ and must not be confused with $\eps_*$ and $\mu_*$. On the other hand, a direct calculation  -- we just repeat all arguments between \eref{Lem:ScalingDe.PEqn1} and \eref{Lem:ScalingDe.PEqn2} with $\pair{\eps_*}{\mu_*}$ instead of $\pair{\eps_{\delta_*}}{\mu_{\delta_*}}$ --  reveals
\begin{eqnarray*}
\at{1-\delta_*}^2&=&\norm{ \hat{V}_{\eps_*}}_2^2=
2\,\frac{\eps_*^2}{\mu_{*}}\int\limits_{-1/\at{2\mu_*}}^{+1/\at{2\mu_*}}\tilde{W}_0\at{\tilde{x}}^2\dint{\tilde{x}}\\&=&
\Bat{1+\eps_{*}\at{\ol{\ka}-1}}
\at{1-\eps_{*}\at{1+\frac{\ol{\eta}}{\ol\mu}}}+\DO{\eps_{*}^m}\,.
\end{eqnarray*}
Equating the right hand sides in both identities we then conclude
\begin{equation*}
\eps_{\delta_*}=\eps_*+\DO{\eps_*^m}\,,\qquad \mu_{\delta_*}=\mu_*+\DO{\eps_*^m}\,,
\end{equation*}
where the last identity holds due to the $\eps_*$-dependence of $\mu_*$ and Lemma \ref{Lem:ScalingMu}, which provides an approximation of $\mu_{\delta_*}$ in terms of $\eps_{\delta_*}$. Finally, exploiting the properties of $\tilde{S}_0$, $\tilde{W}_0$, and $\tilde{T}_0$
as in the proof of Theorem \ref{Thm:GlobApp} we arrive at
\begin{equation*}
\norm{\hat{R}_{\eps_*} -\hat{R}_{\eps_{\delta_*}}}_\infty+
\norm{\hat{V}_{\eps_*} -\hat{V}_{\eps_{\delta_*}}}_\infty
=\DO{\frac{\abs{\eps_*-\eps_{\delta_*}}}{\eps_*}}=\DO{\eps_*^{m-1}}\,,
\end{equation*}
and obtain analogous estimates for the $q$-norms due to the compactness of the supports. The assertion is now provided by Theorem \ref{Thm:GlobApp}.
\end{proof}
%
%
% ------------------------------------------------------------------
\subsection{On the asymptotic eigenvalue problem}\label{sect:eigenproblem}
% ------------------------------------------------------------------
%
%
Of particular interest in the analysis of solitary waves is the spectrum of the linearized equation. The problem consists of finding eigenpairs $\pair{\la}{U}\in\Rset\times\fspaceL^2\at\Rset$ such that
\begin{equation}
\label{Eqn:Eigenproblem}
\la U=L_\delta U\,,\qquad L_\delta U:=A Q_\delta A U\,,\qquad Q_\delta\at{x}:=
\frac{\Phi^{\prime\prime}\bat{{R}_\delta\at{x}}}{\si_\delta}\,,
\end{equation}
where the function $Q_\delta$ becomes singular in the limit $\delta\to0$, see Figure \ref{Fig.Eigenproblem}. Due to the shift symmetry of the traveling wave equation \eref{Eqn:TW.Int}, there is always the solution
\begin{equation*}
{\la}=1\,,\qquad U=V_\delta^\prime\,,
\end{equation*}
and a natural question is whether this eigenspace is simple or not. In fact, simplicity
would immediately imply some local uniqueness for solitary waves
and is also an important ingredient for both linearized and orbital stability.
\par
Unfortunately,
very little is known about the solution set of \eref{Eqn:Eigenproblem} due to the nonlocality of the operator $A$. In the small-energy limit of FPU-type chains, the corresponding problem has been solved in \cite{FP99} by showing that the spectral properties of the analogue to $L_\delta$ are governed by an asymptotic ODE problem which stems from the KdV equation and admits explicit solutions. The hope is that the asymptotic formulas derived in this paper provide spectral control in the high-energy limit. A detailed study of the singular perturbation problem \eref{Eqn:Eigenproblem} is beyond the
scope of this paper but preliminary investigations indicate that the spectrum of $L_\delta$ depends in the limit $\delta\to0$ -- and at least for sufficiently large $m$ -- crucially on the coefficients $c_{\pm1}$ that are derived in the following result.
\begin{figure}[t!]
\centering{%
\includegraphics[width=0.45\textwidth]{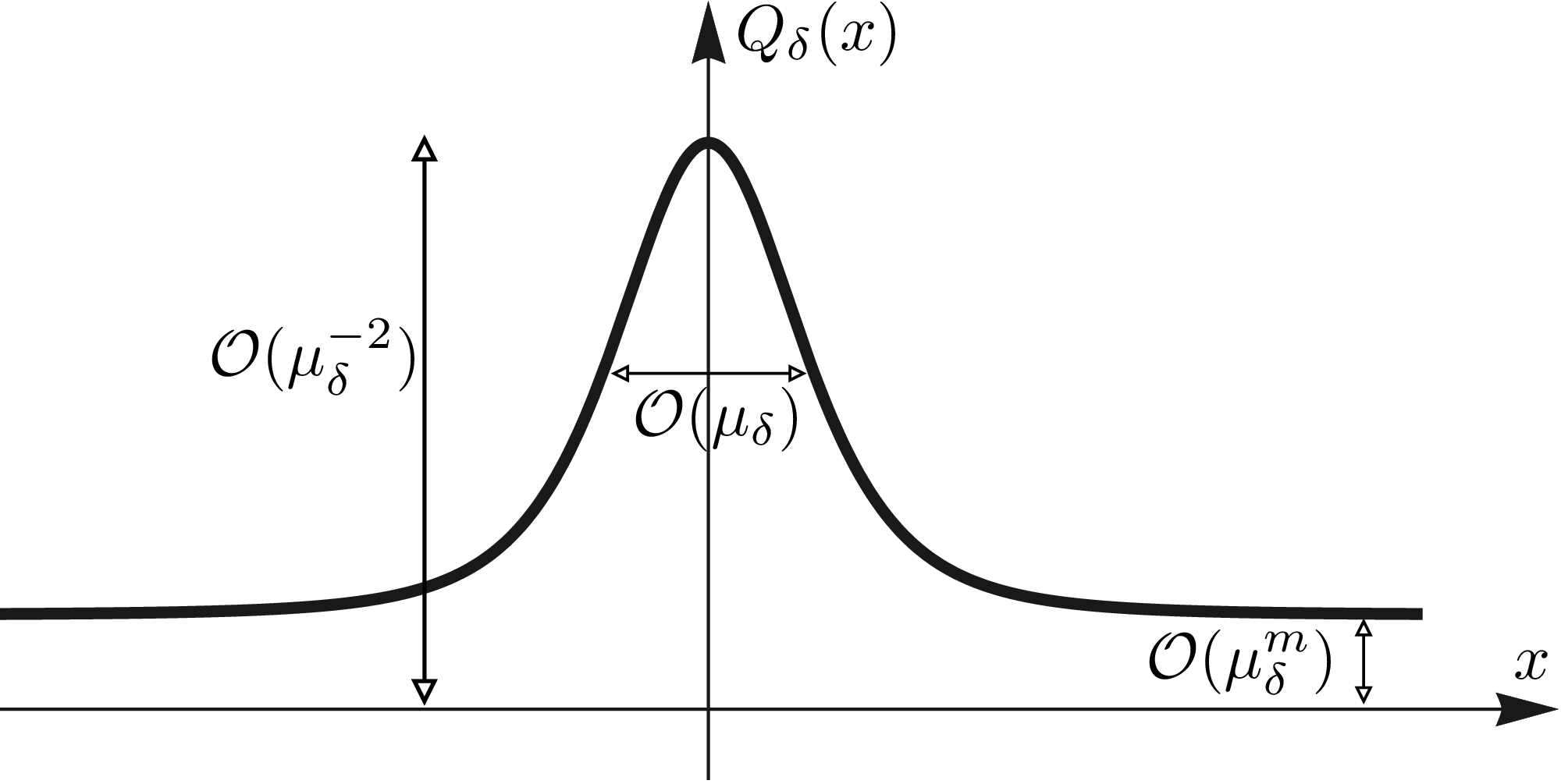}%
}%
\caption{Cartoon of the coefficient function $Q_\delta$ in the linear eigenproblem \eref{Eqn:Eigenproblem}. }%
\label{Fig.Eigenproblem}%
\end{figure}
\begin{theorem}[weak$\star$-expansion of $Q_\delta$]
For any sufficiently regular test function $\varphi$ we have
\begin{equation*}
\int_\Rset Q_\delta\at{x}\varphi\at{x} \dint{x} =
c_{-1}\mu_\delta^{-1}\varphi\at{0}+c_{+1}\mu_\delta^{+1}\varphi^{\prime\prime}\at{0}+\DO{\mu_\delta^{\min{\{m-2,3\}}}}
\end{equation*}
with
\begin{equation*}
c_{-1}:=\int_{\Rset}
\frac{1}
{\at{1+\tilde{S}_0\at{\tilde{x}}}^{m+2}}\dint{\tilde{x}} \,,\qquad
c_{+1}:=\mbox{$\frac12$}\int_{\Rset}
\frac{\tilde{x}^2}
{\at{1+\tilde{S}_0\at{\tilde{x}}}^{m+2}} \dint{\tilde{x}}\,,
\end{equation*}
where the error terms depend on $\varphi$.
\end{theorem}
\begin{proof}
Due to $\Phi^{\prime\prime}\at{r}=\at{1-r}^{-m-2}$ and the scaling relations \eref{Cor:TS.ConvPrms.Eqn1} we find
\begin{equation*}
\int_{\Rset\setminus\ccinterval{-1/2}{+1/2}} Q_\delta\at{x}\varphi\at{x} \dint{x}=\DO{\si_\delta^{-1}}=\DO{\mu_\delta^m}\,,
\end{equation*}
and \eref{Eqn:TipS.Def1} along with \eref{Eqn:DefEpsMu} implies
\begin{eqnarray*}
\int_{\ccinterval{-1/2}{+1/2}} Q_\delta\at{x}\varphi\at{x} \dint{x}&=&\frac{\mu_\delta}{\si_\delta}\int_{J_\delta} \Phi^{\prime\prime}\bat{R_\delta\at{\mu_\delta\tilde{x}}}\varphi\at{\mu_\delta\tilde{x}} \dint{\tilde{x}}\\&=&
\frac{1}{\mu_\delta}\int_{J_\delta}
\frac{\varphi\at{\mu_\delta\tilde{x}}}
{\at{1+\tilde{S}_\delta\at{\tilde{x}}}^{m+2}}
\dint{\tilde{x}}\,.
\end{eqnarray*}
Theorem \ref{Thm:TS.Asymptotics} as well as the linear growth of $\tilde{S}_0$ -- see Lemma \ref{Lem:TS.LimitProfile} -- ensure
\begin{eqnarray}
\label{Lem:EP:Formulas.PEqn1}
\eqalign{ %
\int_{J_\delta}
\frac{\varphi\at{\mu_\delta\tilde{x}}}
{\at{1+\tilde{S}_\delta\at{\tilde{x}}}^{m+2}}
\dint{\tilde{x}}
&= %
\int_{J_\delta}
\frac{\varphi\at{\mu_\delta\tilde{x}}}
{\at{1+\tilde{S}_0\at{\tilde{x}}}^{m+2}}
\dint{\tilde{x}}+\DO{\mu_\delta^{m-1}}
\cr %
&= %
\int_{\Rset}
\frac{\varphi\at{\mu_\delta\tilde{x}}}
{\at{1+\tilde{S}_0\at{\tilde{x}}}^{m+2}}
\dint{\tilde{x}}+\DO{\mu_\delta^{m-1}}
} %
\end{eqnarray}
and by smoothness of $\varphi$ and evenness of $\tilde{S}_0$ we can approximate
\begin{equation*}
\int_{\Rset}
\frac{\varphi\at{\mu_\delta\tilde{x}}}
{\at{1+\tilde{S}_0\at{\tilde{x}}}^{m+2}}
\dint{\tilde{x}}=c_{-1}\varphi\at{0}+c_{+1}\mu_\delta^2\varphi^{\prime\prime}\at{0}+\DO{\mu_\delta^4}\,.
\end{equation*}
The claim now follows by combining all partial estimates from above.
\end{proof}
For completeness we mention that the estimate \eref{Lem:EP:Formulas.PEqn1} is not optimal for moderate values of $m$ and might be improved for the prize of more technical effort.
%
%
%
%
% -----------------------------------------------------------------------------
% - thanks
% -----------------------------------------------------------------------------
%
\section*{Acknowledements}
The authors are grateful for the support by the \emph{Deutsche Forschungsgemeinschaft} (DFG individual grant HE 6853/2-1) and the \emph{London Mathematical Society} (LMS Scheme 4 Grant, Ref~41326).
%
% -----------------------------------------------------------------------------
% - Bibliography
% -----------------------------------------------------------------------------
%
\section*{References}
\end{document}